\documentclass[10pt]{amsart}

\usepackage{amsmath,amsfonts,amssymb,amscd,verbatim}
\usepackage[margin=0.4in]{geometry}
\usepackage{fullpage}

\usepackage{enumerate}

\usepackage{bm}
\usepackage[usenames, dvipsnames]{xcolor}
\usepackage{enumitem}
\usepackage[all]{xy}
\usepackage{mathabx}
\usepackage{array}
\usepackage{tikz,tikz-cd}
\usepackage{multicol}
\usepackage{stmaryrd}

\usepackage{hyperref}
\hypersetup{
	colorlinks,
	linkcolor=[HTML]{0000D5},
	urlcolor=[HTML]{0000D5},
	citecolor=[HTML]{1855CC}
}

\usepackage{float}
\restylefloat{table}

\usepackage{caption} 
\usepackage{tikz}
\usetikzlibrary{arrows}
\tikzset{edge/.style = {->,> = latex'}}

\newcolumntype{C}{>{$}c<{$}}

\numberwithin{equation}{section}
\theoremstyle{plain}
 
\newtheorem{theorem}[equation]{Theorem} 

\newtheorem{corollary}[equation]{Corollary} 
\newtheorem{lemma}[equation]{Lemma}
\newtheorem{proposition}[equation]{Proposition}

\theoremstyle{definition}
\newtheorem{definition}[equation]{Definition}

\newtheorem{hypothesis}[equation]{Hypothesis} 

\newtheorem{remark}[equation]{Remark}

\DeclareMathOperator\Aut{Aut}

\DeclareMathOperator\gr{gr}

\DeclareMathOperator{\lcm}{lcm}

\newcommand\ZZ{\mathbb Z}

\newcommand\cF{\mathcal F}
\newcommand\cQ{\mathcal{Q}}

\newcommand\fsl{\mathfrak sl}

\newcommand\kk{\Bbbk}

\renewcommand{\int}{\mathrm{int}}
\newcommand\inv{^{-1}}

\newcommand\iso{\cong}

\newcommand\tensor{\otimes}

\newcommand\An{\widetilde{A_{n-1}}}

\renewcommand\mod{\mathrm{mod}~}

\newcommand\gtaft{T_\lambda(r,m)}

\renewcommand\th{^{\text{th}}}

\newcommand\set[1]{\left\{#1\right\}}

\newcommand*{\defeq}{\mathrel{\vcenter{\baselineskip0.5ex \lineskiplimit0pt\hbox{\scriptsize.}\hbox{\scriptsize.}}}=} 

\usepackage[normalem]{ulem}

\title{Taft algebra actions on preprojective algebras}

\author[Gaddis]{Jason Gaddis}
\address{Department of Mathematics, Miami University, Oxford, Ohio 45056, USA} 
\email{gaddisj@miamioh.edu}

\author[Oswald]{Amrei Oswald}
\address{Department of Mathematics, University of Washington, Seattle, Washington 98195, USA}
\email{amreio@uw.edu}


\begin{document}

\begin{abstract}
We classify actions of generalized Taft algebras on preprojective algebras of extended Dynkin quivers of type $A$. This may be viewed as an extension of the problem of classifying actions on the polynomial ring in two variables. In cases where the grouplike element acts via rotation on the underlying quiver, we compute invariants of the Taft action and, in certain cases, show that the invariant ring is isomorphic to the center of the preprojective algebra.
\end{abstract}

\maketitle

\section{Introduction}

In \cite{KW}, Kinser and Walton classified actions of Taft algebras on path algebras of schurian, loopless quivers. This was accomplished, in large part, by reducing the problem to the study of $\ZZ_n$-minimal quivers, on which the actions could be explicitly stated. The results were then extended to actions of the Frobenius-Lusztig kernel and actions of the Drinfeld double of the Taft algebra. Kinser and the second-named author studied actions on path algebras by further Hopf algebras, including the quantum enveloping algebra of $\fsl_2$, the corresponding quantum borel subalgebra, and generalized Taft algebras \cite{KO}.

On the other hand, Allman classified actions of Taft algebras on the polynomial ring in two variables \cite{All}. Actions of Taft algebras on quantum places were classified in \cite{GWY}. The importance of this particular family comes from the fact that quantum planes at roots of unity are the only Artin--Schelter regular algebras of global dimension two expressing finite-dimensional quantum symmetry. This analysis was extended to actions of bosonizations of quantum linear spaces on quantum affine spaces in \cite{CG}. In recent years, Taft algebra actions have been studied in a variety of other situations as well \cite{B,BM,cline,CGW,GW2}.

The goal of this project is to merge these two lines of study in order to construct actions of (generalized) Taft algebras on preprojective algebras of the extended Dynkin quivers $\An$. When $n=1$, this algebra is isomorphic to the polynomial ring in two variables. Hence, this project can be seen as an extension of the problem of studying Taft algebras in the commutative setting. More generally, this article initiates the study of finite-dimensional quantum symmetry in the context of (non-connected) graded Calabi-Yau algebras of global dimension two.

Section \ref{sec.background} gives basic definitions of quivers and generalized Taft algebras, as well as a summary of results from \cite{KO}. 
The definition of the preprojective algebra $\Pi_Q$ on the extended Dynkin quiver $Q=\An$ is recalled, as are results on the graded automorphisms of $\Pi_Q$. Lemma \ref{lem.descend} gives conditions for a Taft action on $\kk\overline{Q}$ to descend to an action on $\Pi_Q$.

Section \ref{sec.actions} contains the classification of inner faithful actions on $\Pi_Q$ by generalized Taft algebras, assuming assuming $Q$ has at least three vertices. The analysis is split into two parts depending on whether the grouplike element $g$ acts as a rotation or a reflection. In two vertices, the quiver is not schurian and the analysis is different. This case is not considered here and will be reserved for future work.

Finally, Section \ref{sec.invariants} considers invariants of certain actions from the classification. In the case that there is a single $g$ orbit, it is shown that the invariant ring is precisely the center of $\Pi_Q$. This is related to work of Berrizbeitia on invariants of Tafts actions on quivers \cite{berr}.

\subsection*{Acknowledgements} Oswald was supported by a grant from the Simons Foundation Targeted Grant (917524) to the Pacific Institute for the Mathematical Sciences. 

\section{Background}\label{sec.background}

Let $\kk$ be a field containing a root of unity $\lambda$ such that the characteristic of $\kk$ and the order of $\lambda$ are coprime. These hypotheses are also assumed in \cite{KW}. The group of units of $\kk$ is denoted $\kk^\times$.

A quiver $Q$ is a directed graph with a vertex set denoted $Q_0$ and a set of arrows $Q_1$, equipped with maps $s,t:Q_1 \to Q_0$ where for all $a \in Q_1$, $s(a)$ denotes the \emph{source} of $a$ and $t(a)$ denotes its \emph{target}. The quiver $Q$ is \emph{finite} if $Q_0$ and $Q_1$ are finite as sets. The quiver $Q$ is \emph{schurian} if for every pair of vertices $i \neq j$, there is at most one arrow $a$ with source $s(a)=i$ and $t(a)=j$.

A \emph{path} is a sequence of arrows $p=a_1a_2\cdots a_n$ where $t(a_i)=s(a_{i+1})$ for all $i = 1,\hdots,n-1$. The maps $s,t$ extend to paths in a natural way where $s(p)=s(a_1)$ and $t(p)=t(a_n)$. At each vertex $i$ there is a \emph{trivial loop} $e_i$ which satisfies $s(e_i)=i=t(e_i)$. The \emph{path algebra} $\kk Q$ is spanned, as a vector space, by the paths on $Q$, and multiplication of paths $p$ and $q$ is the concatenation $pq$ if $t(p)=s(q)$, and zero otherwise. Let $\kk Q_k$ denote the span of paths of length $k$. There is a natural filtration $\cF=\{F_k\}$ on $\kk Q$ in which $F_k$ is spanned by the set of paths of length at most $k$.

A \emph{quiver automorphism} is a pair of bijections $(f_0,f_1)$ on the sets $Q_0$ and $Q_1$, respectively, such that for all $a \in Q_1$, $(f_0 \circ s)(a) = (s \circ f_1)(a)$ and $(f_0 \circ t)(a) = (t \circ f_1)(a)$. A group $G$ \emph{acts} on the quiver $Q$ if $G$ may be identified with a group of quiver automorphisms. 
The group action \emph{preserves the filtration by path length} if $g(F_i) \subseteq F_i$. If $Q$ is schurian, then this implies that the automorphism giving the action of $g$ is in fact graded, that is, $g(\kk Q_i) \subseteq \kk Q_i$ \cite[Lemmas 3.1, 3.2]{KW}.

\subsection{Hopf actions on path algebras}

Suppose $\lambda$ is a primitive $r\th$ root of unity,
$r>1$, and $m$ a positive multiple of $r$. The corresponding \emph{generalized Taft algebra} is the (finite-dimensional) $\kk$-algebra 
\[ \gtaft := \kk\langle g,x \mid gx-\lambda xg, g^m-1,x^r\rangle.\]
Then $\gtaft$ is a Hopf algebra where the coalgebra
structure is given by 
\begin{align*}
&\Delta(g)=g \tensor g, \quad \Delta(x) = 1 \tensor x + x \tensor g, \qquad
\varepsilon(g) = 1, \quad \varepsilon(x)=0.
\end{align*}
That is, $g$ is grouplike and $x$ is $(1,g)$-skew primitive. The antipode on $\gtaft$ is given by 
$S(g)=g\inv$ and $S(x)=-xg\inv$.
Throughout, $G$ denotes the (cyclic) group generated by $g$.

The following results (Proposition \ref{prop.vertact} and Theorem \ref{thm.quiveract}) are 
adapted from \cite{KO} (Proposition 3.1 and Corollary 3.9). They have been restated below
to avoid discussion of the quantum Borel.

\begin{proposition}\label{prop.vertact}
Let $\lambda$ be a primitive $r^\text{th}$ root of unity.
Let $Q_0$ be the vertex set of a quiver.
\begin{enumerate}[label=(\Alph*),leftmargin=*]
\item The following data determines a Hopf action of $\gtaft$ on $\kk Q_0$.
\begin{enumerate}[label=(\roman*)]
\item[(i)] A permutation action of $G$ on the set $Q_0$ such that $\#(G\cdot i)\mid m$ for all $i \in Q_0$.
\item[(ii)] A collection of scalars $(\gamma_i \in \kk)_{i \in Q_0}$ such that, for all $i \in Q_0$, $\gamma_{g \cdot i} = \lambda\inv \gamma_i$ and
either $\#(G\cdot i)=r$ or $\gamma_i=0$ for all $i$.
\end{enumerate}
Given this data, the $x$-action is given by
\[ x \cdot e_i = \gamma_i e_i - \gamma_i \lambda\inv e_{g \cdot i}\quad\text{for all $i \in Q_0$}.\]

\item Every action of $\gtaft$ on $\kk Q_0$ is of the form above.
\end{enumerate}
\end{proposition}

\begin{definition}\label{defn.qtmap}
Let $Q$ be a quiver. A $\kk$-linear endomorphism 
$\sigma: \kk Q_0 \oplus \kk Q_1 \to \kk Q_0 \oplus \kk Q_1$
is a \emph{quiver-Taft map} for $Q$ if it satisfies
\begin{enumerate}[label=($\sigma$\arabic*)]
\item \label{sig1} $\sigma(\kk Q_0) = 0$,
\item \label{sig2} $\sigma(a) = e_{s(a)} \sigma(a) e_{g \cdot t(a)}$ for all $a \in Q_1$, and
\item \label{sig3} $\sigma(g \cdot a) = \lambda\inv g \cdot \sigma(a)$ for all $a \in Q_1$.
\end{enumerate}
\end{definition}

\begin{theorem}\label{thm.quiveract}
Let $\lambda$ be a primitive $r^\text{th}$ root of unity.
Let $Q$ be the vertex set of a quiver.
\begin{enumerate}[label=(\Alph*),leftmargin=*]
\item The following data determines a Hopf action of $\gtaft$ on $\kk Q$.
\begin{enumerate}[label=(\roman*)]
    \item A Hopf action of $\gtaft$ on $\kk Q_0$ as above.
    \item A representation of $G$ on $\kk Q_1$ satisfying $s(g \cdot a) = g\cdot s(a)$ and $t(g\cdot a) = g\cdot t(a)$ for all $a \in Q_1$,
    and the element $g^m$ acts as the identity on all of $\kk Q$.
    \item A quiver-Taft map $\sigma$ for $Q$ satisfying
    \begin{align}\label{eq.gam_sig}
        \gamma_{s(a)}^r g^r(a) - \gamma_{t(a)}^r a = \sigma^r(a)\quad\text{for all $a \in \kk Q_1$}.
    \end{align}
\end{enumerate}
Given this data, the $x$-action on $a \in Q_1$ is given by
\[ x \cdot a = \gamma_{t(a)} a - \gamma_{s(a)} \lambda\inv(g \cdot a) + \sigma(a).\]

\item Every action of $\gtaft$ on $\kk Q$ is of the form above.
\end{enumerate}
\end{theorem}

In general, the action of a Hopf algebra $H$ on an algebra $A$ is \emph{inner faithful} if there is no Hopf ideal $I$ of $H$ such that $IA = 0$. When $H=\gtaft$, the action is inner faithful if and only if $x \cdot A \neq 0$ \cite[Lemma 2.5]{KW}.

The following is an immediate consequence of the above parametrization.

\begin{corollary}\label{cor.fixed}
Suppose $\gtaft$ acts on $\kk Q$ as in Theorem \ref{thm.quiveract}. 
\begin{enumerate}
    \item If $g\cdot e_i=e_i$, then $\gamma_i=0$ and $e_i$ is a fixed point.   
    \item Suppose that $Q$ is schurian and $g \cdot e_i=e_i$ for all $i$. Then $x \cdot a = 0$ for all $a \in Q_1$. That is, the action of $\gtaft$ on $\kk Q$ is not inner faithful.
\end{enumerate}
\end{corollary}
\begin{proof}
(1) If $g\cdot e_i=e_i$, then $\gamma_{g\cdot i}=\gamma_i=\lambda\inv \gamma_i$, so $\gamma_i=0$. It follows that $x\cdot e_i = 0$.

(2) By part (1), $\gamma_i=0$ for all $i$, so $x \cdot e_i=0$ for all $i$. Since the quiver $Q$ is schurian, there exists at most one unique arrow $a_{ij}$ with $s(a_{ij})=i$ and $t(a_{ij})=j$. Then $g \cdot a_{ij}=\mu_{ij} a_{ij}$ for some $\mu_{ij} \in \kk^\times$. Now 
\[ \sigma(a_{ij}) = e_{s(a_{ij})} \sigma(a_{ij}) e_{t(a_{ij})} = \eta_{ij} a_{ij}\]
for some $\eta_{ij} \in \kk$. Then
\[
\mu_{ij} \eta_{ij} a_{ij}
= \sigma(\mu_{ij} a_{ij})
= \sigma(g \cdot a_{ij}) = \lambda\inv g \sigma(a_{ij}) = \lambda\inv g \cdot (\eta_{ij} a_{ij}) = \lambda\inv \eta_{ij} \mu_{ij} a_{ij}.
\]
Thus, $\eta_{ij}=0$ so $\sigma(a_{ij})=0$. Thus, $x \cdot a_{ij} = 0$.
\end{proof}

\subsection{The preprojective algebra and its automorphisms}

Let $Q$ be a finite quiver. Define $Q_1^*$ by the rule: if $a \in Q_1$, then $a^* \in Q_1^*$ with $s(a^*)=t(a)$ and $t(a^*)=s(a)$.
The \emph{double} of a quiver $Q$, denoted $\overline{Q}$, is then defined by setting $\overline{Q}_0=Q_0$ and $\overline{Q}_1=Q_1 \cup Q_1^*$ 
 
\begin{definition}
The \emph{preprojective algebra} associated to a finite quiver $Q$ is the quotient $\kk \overline{Q}/(\Omega)$ where $\Omega=\sum_{a\in Q_1} aa^*-a^*a \in \overline{Q}$ is the \emph{preprojective relation}.
\end{definition}

For the remainder, let $Q=\An$ with $n \geq 1$. That is, $Q$ is an extended Dynkin quiver of type $A$. Throughout, all indices will be interpreted mod $n$. These quivers are given by
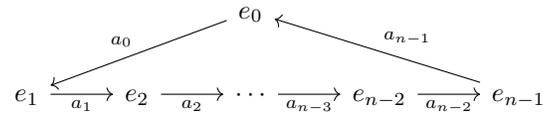
\begin{center}
\begin{tikzcd}
& & e_0 \arrow[lld, "a_0"'] & & \\
e_1 \arrow[r, "a_1"'] & e_2 \arrow[r, "a_2"'] & \cdots \arrow[r, "a_{n-3}"'] & e_{n-2} \arrow[r, "a_{n-2}"'] & e_{n-1} \arrow[llu, "a_{n-1}"']
\end{tikzcd}
\end{center}
The vertex set $\overline{Q_0}$ of $\overline{Q}$ is $\{e_0,\dots,e_{n-1}\}$, and there is exactly one nonstar arrow, $a_i$, from $e_i$ to $e_{i+1}$ and one star arrow, $a_i^*$, from $e_{i+1}$ to $e_{i}$ for each $i=0,\dots,n-1$. 
Specializing the defining relation to each vertex gives the relations
\begin{align}\label{eq.preproj}
\Omega_i = a_i^*a_i - a_{i+1}a_{i+1}^* \qquad\text{for each $i$}.
\end{align}
When $n=1$, then $\Pi_Q \iso \kk[x,y]$. For $n \geq 3$, $Q$ is \emph{schurian}.

It is well-known that $\Pi_Q$ is a locally finite noetherian PI algebra. Moreover, $\Pi_Q$ is Calabi--Yau of global dimension and Gelfand-Kirillov dimension two. The invariant theory of $\Pi_Q$ was studied by Weispfenning \cite{SW1,SW2}. The work in \cite{BKG} considers Auslander's Theorem for for dihedral actions on $\Pi_Q$. The goal herein is to classify actions of Taft algebras on $\Pi_Q$ and their invariants.

The graded automorphism group of $\Pi_Q$ was given in \cite{SW1}. In particular, \cite[Proposition 3.6]{SW1} gives that $\Aut_{\gr}(\Pi_Q) \iso D_n \ltimes N$ where $N$ is the subgroup of graded automorphisms which fix each vertex. It will be useful to recount the key part of this argument below.

\begin{lemma}\label{lem.gact}
Suppose $n \geq 3$ and let $g$ be a graded automorphism of $\kk\overline{Q}$.
The action of $g$ on $\kk\overline{Q}$ is described by one of the following cases where $d$ is an integer satisfying $0 \leq d \leq n-1$:
\begin{enumerate}
\item ($g$ acts as a rotation on $\overline{Q}$) 
There exists $\mu_i,\mu_i^* \in \kk^\times$ such that for all $0\leq i \leq n-1$,
\[  g\cdot e_{i}=e_{i+d}, \qquad 
    g\cdot a_i=\mu_ia_{i+d}, \qquad
    g\cdot a_i^*=\mu_i^* a_{i+d}^*. \]
\item ($g$ acts as a reflection on $\overline{Q}$)
There exists $\mu_i,\mu_i^* \in \kk^\times$ such that for all $0\leq i \leq n-1$,
\[  g\cdot e_i = e_{n-(d+i)}, \qquad 
    g\cdot a_i = \mu_i a_{n-(d+i+1)}^*, \qquad 
    g\cdot a_i^* = \mu_i^* a_{n-(d+i+1)}.
\]
\end{enumerate}
\end{lemma}
\begin{proof}
(1) If $g\cdot a_i = \mu_i a_j$ where $\mu_i\in\kk^\times$, then it must be the case $g\cdot e_i=e_j$ and $g\cdot e_{i+1}=e_{j+1}$. Therefore, $g\cdot a_i^*=\mu_i^*a_j^*$. Since $s(g\cdot a_{i+1})=j+1$, then $g\cdot a_{i+1}=\mu_{i+1}a_{j+1}$ and $g\cdot a_{j+1}^*=\mu_{i+1}^*a_{j+1}^*$ and so on for $a_{i+k}$ and $a_{i+k}^*$ for any $1 \leq k \leq n-1$. Thus, if $g\cdot e_{0}=e_{d}$ and $g\cdot a_0=\mu_ia_{d}$ for some $0\leq d \leq n-1$, then $g\cdot e_{i}=e_{i+d}$, $g\cdot a_i=\mu_ia_{i+d}$, and $g\cdot a_i^*=\mu_i^* a_{i+d}^*$ for all $0\leq i \leq n-1$.

(2) If $g\cdot a_i=\mu_ia_j^*$ where $\mu_i\in\kk^\times$, then it must be the case that $g\cdot e_i=e_{j+1}$ and $g\cdot e_{i+1}=e_{j}$. Therefore, $g\cdot a_i^*=\mu_i^*a_j$. Then, $s(g\cdot a_{i+1})=j$, and so $g\cdot a_{i+1}=\mu_{i+1} a_{j-1}^*$ since this is the only arrow left with source $j$. Similarly, $g\cdot a_{i+1}^*=\mu_{i+1}^* a_{j-1}$. In general, if $g\cdot e_0 = e_{n-d}$ and $g\cdot a_0=\mu_0a_{n-(d+1)}^*$, then $g\cdot e_i = e_{n-(d+i)}$, $g\cdot a_i = \mu_i a_{n-(d+i+1)}^*$ and $g\cdot a_i^* = \mu_i^* a_{n-(d+i+1)}$ for all $0\leq i \leq n-1$.
\end{proof}

Suppose $\gtaft$ acts inner faithfully on $\kk\overline{Q}$ such that 
$g$ acts by rotation. By Lemma \ref{lem.gact}(1) and Corollary \ref{cor.fixed}, $d \neq 0$.
The final result of this section gives conditions for an action on $\kk\overline{Q}$ to pass to the quotient $\Pi_Q$.

\begin{lemma}\label{lem.descend}
Suppose $n \geq 3$. 
An inner faithful action of $\gtaft$ on $\kk\overline{Q}$
descends to an action on $\Pi_Q$ if and only if
\begin{align}
\label{eq.mu_rel}
    0 &= \mu_i\mu_i^*-\mu_{i+1}\mu_{i+1}^*, \\
\label{eq.sigrel}
    0& = a_i^*\sigma(a_i) + \sigma(a_i^*)(g \cdot a_i) - a_{i+1}\sigma(a_{i+1}^*) - \sigma(a_{i+1})(g \cdot a_{i+1}^*).
\end{align}
\end{lemma}
\begin{proof}
Each $\Omega_i$ (see \eqref{eq.preproj}) is homogeneous with $s(\Omega_i)=t(\Omega_i)=i+1$.
Since $g$ acts as an automorphism of $\kk\overline{Q}$,
then $g(\Omega_i)$ is homogeneous with $s(g(\Omega_i))=t(g(\Omega_i))=g \cdot (i+1)$.
By the discussion above, the image of $a_i^*a_i$ under $g$ consists of one starred arrow and one unstarred arrow, 
and the image has source and target $e_{g \cdot (i+1)}$.
In $\Pi_Q$ there is a unique such path $p$,
so $g(a_i^*a_i) = z p$ for some $z \in \kk$.
The same argument applies to the image of $a_{i+1}a_{i+1}^*$, so $g(a_{i+1}a_{i+1}^*) = z p$. 
It follows from Lemma \ref{lem.gact} that
\[
0 = g \cdot (a_i^*a_i - a_{i+1}a_{i+1}^*)
    = \mu_i\mu_i^* z p -\mu_{i+1}\mu_{i+1}^* z p
    = z (\mu_i^*\mu_i-\mu_{i+1}\mu_{i+1}^*)p.
\]
This proves \eqref{eq.mu_rel}. 
From Theorem \ref{thm.quiveract},
\begin{align*}
x &\cdot (a_i^*a_i - a_{i+1}a_{i+1}^*) \\
    &= \left(a_i^* (\gamma_{i+1} a_i - \gamma_i \lambda\inv (g \cdot a_i) + \sigma(a_i))
        + (\gamma_i a_i^* - \gamma_{i+1} \lambda\inv (g \cdot a_i^*)+\sigma(a_i^*))(g \cdot a_i) \right) \\
    &\qquad- \left( a_{i+1}(\gamma_{i+1} a_{i+1}^* - \gamma_{i+2} \lambda\inv (g \cdot a_{i+1}^*) + \sigma(a_{i+1}^*))
        + (\gamma_{i+2} a_{i+1} - \gamma_{i+1} \lambda\inv (g \cdot a_{i+1})+\sigma(a_{i+1}))(g \cdot a_{i+1}^*)\right) \\
    &= \gamma_{i+1}\left( (a_i^*a_i - a_{i+1}a_{i+1}^*)
        - \lambda\inv g \cdot (a_i^*a_i - a_{i+1}a_{i+1}^*)\right)
    + (1-\lambda\inv)\left(\gamma_ia_i^*(g \cdot a_i) - \gamma_{i+2}a_{i+1}(g \cdot a_{i+1}^*)\right)  \\
    &\qquad+ \left(a_i^*\sigma(a_i) + \sigma(a_i^*)(g \cdot a_i) - a_{i+1}\sigma(a_{i+1}^*) - \sigma(a_{i+1})(g \cdot a_{i+1}^*) \right) \\
    &= (1-\lambda\inv)\left(\gamma_ia_i^*(g \cdot a_i) - \gamma_{i+2}a_{i+1}(g \cdot a_{i+1}^*)\right) \\
        &\qquad + \left(a_i^*\sigma(a_i) + \sigma(a_i^*)(g \cdot a_i) - a_{i+1}\sigma(a_{i+1}^*) - \sigma(a_{i+1})(g \cdot a_{i+1}^*) \right).
\end{align*}
It now suffices to show that
\[ \gamma_ia_i^*(g \cdot a_i) - \gamma_{i+2}a_{i+1}(g \cdot a_{i+1}^*) = 0.\]

Suppose $g$ is a rotation. Then $g \cdot a_i = \mu_i a_{i+d}$. Hence, $a_i^*(g \cdot a_i) \neq 0$ if and only if $d=0$, in which case the action is not inner faithful. Similarly, $g \cdot a_{i+1}^* = \mu_{i+1}^* a_{i+1+d}^*$, so again $a_{i+1}(g \cdot a_{i+1}^*) \neq 0$ if and only if $d=0$.

Now suppose $g$ is a reflection. 
Since $g \cdot a_i = \mu_i a_{n-(d+i+1)}^*$, 
then $a_i^*(g \cdot a_i) \neq 0$ precisely when $n-(d+i+1)=i-1$, 
so $i$ is a fixed point and $\gamma_i=0$.
Similarly, $a_{i+1}(g \cdot a_{i+1}^*) \neq 0$ if and only if $i+2$ is a fixed point and $\gamma_{i+2}=0$. 
This leaves \eqref{eq.sigrel}, which is a relation from $i+1$ to $g \cdot (i+1)$.
\end{proof}

\section{Taft actions on preprojective algebras}
\label{sec.actions}

Throughout this section, assume that $Q=\An$ with $n \geq 3$,
$m$ and $r$ are positive integers so that $r\mid m$, and $\lambda$ is a primitive $r\th$ root of unity. Let $g$ and $x$ be the grouplike and $(1,g)$-skew primitive generators of $\gtaft$, respectively. 
All actions will be assumed to be inner faithful and linear.
Hence, by Theorem \ref{thm.quiveract}, $g$ acts as an automorphism on $\overline{Q}_1$. In light of Lemma \ref{lem.gact}, the analysis will be split between $g$ a rotation automorphism (Section \ref{sec.rotation}) and $g$ a reflection automorphism (Section \ref{sec.reflection}).
Given an action of $g$ on $\kk Q_0$, the action on $\kk\overline{Q}$ (and then on $\Pi_Q$) is parameterized by 
\begin{itemize}
\item the choice of integers $r,m$ defining $T_\lambda(r,m)$, 
\item the choice of scalars $(\gamma_i \in \kk)_{i \in Q_0}$ defined in Proposition \ref{prop.vertact}, 
\item the choice of scalars $(\mu_i,\mu_i^* \in \kk^\times)_{i \in Q_0}$ such that $g \cdot a_i = \mu_ia_{g \cdot i}$ and $g \cdot a_i^* = \mu_i^*a_{g \cdot i}^*$, and
\item the choice of a quiver-Taft map $\sigma$.
\end{itemize}
Given this information, the action extends to $\kk\overline{Q}$ and $\Pi_Q$ by
\begin{align*}
g \cdot e_i &= e_{g \cdot i} &
x \cdot e_i &= \gamma_i e_i - \gamma_i \lambda\inv e_{g \cdot i} \\
g \cdot a_i &= \mu_i a_{g \cdot i} &
x \cdot a_i &= \gamma_{i+1} a_i - \gamma_{i} \lambda\inv g\cdot a_{i} + \sigma(a_i) \\
g \cdot a_i^* &= \mu_i^* a_{g \cdot i}^* &
x \cdot a_i^* &= \gamma_i a_i^* - \gamma_{i+1} \lambda\inv g\cdot a_{i}^* + \sigma(a_i^*).
\end{align*}
The results below give necessary conditions in each case (rotation or reflection) so that there is an action of $\gtaft$ on $\Pi_Q$.
That these conditions are sufficient for an action will be clear.

\subsection{Rotation actions on the preprojective algebra}
\label{sec.rotation}

Throughout this section, assume that $\gtaft$ acts linearly and inner faithfully on $\kk\overline{Q}$ and the corresponding preprojective algebra $\Pi_Q$, and that there exists some $d \in \ZZ$, $0< d \leq n-1$, such that $g\cdot e_i=e_{i+d}$ for all $i \in Q_0$.
By Proposition \ref{prop.vertact}, $\gamma_{i+d} = \lambda\inv \gamma_i$, so
\begin{align}\label{eq.gam_rel}
\gamma_{i+kd}= \lambda^{-k}\gamma_i \qquad\text{for $0\leq i < d$ and $0\leq k <r$.}
\end{align}
Let $\kappa$ and $\tau$ denote the size and number of $G$ orbits of vertices, respectively. It follows that for all $i\in Q_0$,
\begin{equation*}
    \kappa\defeq\#(G\cdot i)=\frac{n}{\gcd{(n,d)}},\quad \tau\defeq \frac{n}{\kappa}=\gcd{(n,d)}.
\end{equation*}
Since $g^m\cdot a = a$ for all $a\in Q_1$, then for $0\leq i < \tau$,
\begin{equation}\label{eq.muprod-1root}
\mu_i\mu_{i+d}\dotsc\mu_{i+(\kappa-1)d}=\zeta_i\quad \text{and} \quad \mu_i^*\mu_{i+d}^*\dotsc\mu_{i+(\kappa-1)d}^*=\zeta_i^*,
\end{equation}
where $\zeta_i$ and $\zeta_i^*$ are $\left(\frac{m}{\kappa}\right)^\text{th}$ roots of unity.

The next lemma establishes initial properties of the quiver-Taft map $\sigma$
and these properties, as with many of the arguments below, will vary
depending on the choice of $d$.

\begin{lemma}\label{lem.rot_sigma} 
Suppose $\gtaft$ acts linearly and inner faithfully on $\kk\overline{Q}$ so that $g$ acts on $\kk\overline{Q}$ via a rotation automorphism with $g\cdot e_i=e_{i+d}$ for some $0 < d \leq n-1$. 
If $\sigma: \kk Q_0 \oplus \kk Q_1 \to \kk Q_0 \oplus \kk Q_1$ is the quiver-Taft map corresponding to this Taft action, then there exist $c_0,c_1,c_0^*,c_1^*\in\kk$ so that
\begin{align}\label{eq.rot_sigma1}
\sigma(a_i) &= \begin{cases}
    \lambda^{i/2} \left( \frac{\mu_i\mu_{i-2} \cdots \mu_2}{\mu_{i-1}^*\mu_{i-3}^*\cdots \mu_1^*} \right) c_0 a_{i-1}^* & \text{if $d=n-2$ and $i$ is even} \\
    \lambda^{(i-1)/2} \left( \frac{\mu_i\mu_{i-2} \cdots \mu_3}{\mu_{i-1}^*\mu_{i-3}^*\cdots \mu_2^*} \right) c_1 a_{i-1}^* & \text{if $d=n-2$ and $i$ is odd} \\
    \lambda^i(\mu_i\mu_{i-1}\cdots\mu_1)c_0 e_i & \text{if $d=n-1$} \\
    0       & \text{if $d \neq n-1,n-2$},   
\end{cases}\\[.5em]
\label{eq.rot_sigma2}
\sigma(a_i^*) &= \begin{cases}
     \lambda^{-i/2} \left( \frac{\mu_{i-1}\mu_{i-3} \cdots \mu_1}{\mu_{i-2}^*\mu_{i-4}^*\cdots \mu_0^*} \right) c_0^* a_{i+1} & \text{if $d=2$ and $i$ is even} \\
    \lambda^{-(i-1)/2} \left( \frac{\mu_{i-1}\mu_{i-3} \cdots \mu_2}{\mu_{i-2}^*\mu_{i-4}^*\cdots \mu_1^*} \right) c_1^* a_{i+1} & \text{if $d=2$ and $i$ is odd} \\
    (\lambda^i\mu_{i-1}^*\mu_{i-2}^*\cdots\mu_0^*)\inv c_0^* e_{i+1} & \text{if $d=1$} \\
    0       & \text{if $d \neq 1,2$}.
\end{cases} 
\end{align}
\end{lemma}
\begin{proof}
By \ref{sig2},
\begin{align*}
\sigma(a_i) 
    &= e_{s(a_i)} \sigma(a_i) e_{g \cdot t(a_i)}
    = e_i \sigma(a_i) e_{i+d+1} \\
\sigma(a_i^*) 
    &= e_{s(a_i^*)} \sigma(a_i^*) e_{g \cdot t(a_i^*)}
    = e_{i+1} \sigma(a_i^*) e_{i+d}.
\end{align*}
Since $n > 2$, then $\sigma(a_i)\neq 0$ implies $i-1\leq i+d+1\leq i+1 $, so $n-2 \leq d \leq n$. Similarly, $\sigma(a_i^*)\neq 0$ implies $i\leq i+d\leq i+2 $, so $0 \leq d \leq 2$. It follows that
\begin{align}\label{eq.rot_sigma0}
\sigma(a_i) = \begin{cases}
    c_i a_{i-1}^* & \text{ if $d=n-2$} \\
    c_i e_i & \text{ if $d=n-1$} \\
    0       & \text{ otherwise},   
\end{cases} \qquad 
\sigma(a_i^*) = \begin{cases}
    c_i^* a_{i+1} & \text{ if $d=2$} \\
    c_i^* e_{i+1} & \text{ if $d=1$} \\
    0       & \text{ otherwise},   
\end{cases} 
\end{align}
where $c_i,c_i^* \in \kk$ for all $i \in Q_0$. 

Since $\sigma$ is a quiver-Taft map for $Q$, then by \ref{sig3},
$\sigma(g\cdot a_i)=\lambda^{-1}g\cdot\sigma(a_i)$.
A case-by-case analysis using this along with \eqref{eq.rot_sigma0} gives the following:
\begin{align*}
&(d=n-1) & 
    \mu_ic_{i-1} e_{i-1}
        &= \mu_i\sigma(a_{i-1})
        = \sigma(g \cdot a_i)
        = \lambda\inv g \cdot \sigma(a_i)
        = \lambda\inv g \cdot (c_ie_i)
        = \lambda\inv c_ie_{i-1} \\
&(d=1) & 
    \mu_i^*c_{i+1}^* e_{i+2} 
        &= \mu_i^*\sigma(a_{i+1}^*) 
        = \sigma(g \cdot a_i^*) 
        = \lambda\inv g \cdot \sigma(a_i^*)
        = \lambda\inv g\cdot (c_i^*e_{i+1}) 
        = \lambda\inv c_i^*e_{i+2} \\
&(d=n-2) & 
    \mu_i c_{i-2} a_{i-3}^*
        &= \sigma(\mu_i a_{i-2})
        = \sigma(g \cdot a_i) 
        = \lambda\inv g\cdot \sigma(a_i)
        = \lambda\inv g\cdot (c_ia_{i-1}^*)
        = \lambda\inv c_i \mu_{i-1}^*a_{i-3}^* \\
&(d=2) & 
    \mu_i^*c_{i+2}^* a_{i+3} 
        &= \mu_i^*\sigma(a_{i+2}^*) 
        = \sigma(g \cdot a_i^*) 
        = \lambda\inv g \cdot \sigma(a_i^*)
        = \lambda\inv g\cdot (c_i^*a_{i+1}) 
        = \lambda\inv c_i^*\mu_{i+1}a_{i+3}.
\end{align*}
Thus, the $c_i,c_i^*$ satisfy the following:
\begin{equation}\label{eq.rot_c}
c_i = \begin{cases}
    \lambda\mu_ic_{i-1} & \text{if $d=n-1$} \\
    \lambda\frac{\mu_i}{\mu_{i-1}^*}c_{i-2} & \text{if $d=n-2$},
\end{cases} \qquad
c_i^*
    = \begin{cases}
    (\lambda\mu_{i-1}^*)\inv c_{i-1}^* & \text{if $d=1$} \\
    \lambda\inv \frac{\mu_{i-1}}{\mu_{i-2}^*}c_{i-2}^* & \text{if $d=2$.}
\end{cases}
\end{equation}

Now, in the case $d=n-1$,
\begin{equation*}
c_i = \lambda\mu_i c_{i-1} = \lambda^2\mu_i\mu_{i-1} c_{i-2}= \cdots =\lambda^i(\mu_i\mu_{i-1} \cdots \mu_1)c_0.
\end{equation*}
The other cases are similar. 
\end{proof}

The next lemma establishes restrictions on the 
$\gamma_i$ and their relation to the $\mu_i,\mu_i^*$.

\begin{lemma}\label{lem.rot_gamma}
Suppose $\gtaft$ acts linearly and inner faithfully on $\kk\overline{Q}$ so that $g$ acts on $\kk\overline{Q}$ via a rotation automorphism with $g\cdot e_i=e_{i+d}$ for some $0 < d \leq n-1$. 
Suppose there exists $j\in Q_0$ with $\gamma_j\neq 0$. Then $r=\kappa$. If $n \neq 4$ or $d \neq 2$, then the following hold:
\begin{enumerate}
 \item \label{rot_gam1}
    $\gamma_{i+1}^r 
	       = (\mu_i \mu_{i+d} \cdots \mu_{i+(r-1)d})\gamma_i^r
    	   =(\mu_i^* \mu_{i+d}^* \cdots \mu_{i+(r-1)d}^*)\inv\gamma_i^r$ 
    or equivalently, $\gamma_{i+1}^r=\zeta_i\gamma_i^r=(\zeta_i^*)\inv\gamma_i^r$,
    with $\zeta_i,\zeta_i^*$ as in \eqref{eq.muprod-1root},
    \item \label{rot_gam2} $\gamma_i\neq 0$ for all $i$,
    \item \label{rot_gam3} $\zeta_i^*=\zeta_i\inv$,
    \item \label{rot_gam4} $\mu_0\mu_1 \cdots \mu_{n-1} = \mu_0^* \cdots \mu_{n-1}^*=1$,
    \item \label{rot_gam5} $m=n$ if $\gcd(n,d)=1$.
\end{enumerate}
\end{lemma}
\begin{proof}
Note that $\#(G\cdot i)=\frac{n}{\gcd{(n,d)}}$ for every $i \in \overline{Q}_0$.
Then $\#(G\cdot j)=\frac{n}{\gcd(n,d)}=r$ by Proposition \ref{prop.vertact}.
Let $s=\frac{n}{r}$ so that $s$ counts the number of (distinct) $G$-orbits of $Q_0$.

\eqref{rot_gam1} 
Since $n\neq4$ or $d\neq 2$, $\sigma^r(a_i)=\sigma^r(a_i^*)=0$ for all $i \in Q_0$. Thus, by \eqref{eq.gam_sig},
\begin{align*}
0   &= \sigma^r(a_i) 
    = \gamma_i^r g^r(a_i) - \gamma_{i+1}^r a_i
    = (\gamma_i^r \mu_i \mu_{i+d} \cdots \mu_{i+(r-1)d} - \gamma_{i+1}^r) a_i \\
0   &= \sigma^r(a_i^*)
    = \gamma_{i+1}^r g^r(a_i^*) - \gamma_i^r a_i^*
    = (\gamma_{i+1}^r\mu_i^*\mu_{i+d}^*\cdots\mu_{i+(r-1)d}^*-\gamma_i^r)a_i^*.
\end{align*}

\eqref{rot_gam2}
This follows directly from \eqref{rot_gam1}.

\eqref{rot_gam3} This follows directly from \eqref{rot_gam1} and \eqref{rot_gam2}.

\eqref{rot_gam4} Using \eqref{rot_gam1},
\begin{align*}
\gamma_0^r(\mu_0 \mu_1 \cdots \mu_{n-1})
    &= \gamma_0^r\prod_{i=0}^{s-1} \mu_i \mu_{i+d} \cdots \mu_{i+(r-1)d}
    = \gamma_1^r\prod_{i=1}^{s-1} \mu_i \mu_{i+d} \cdots \mu_{i+(r-1)d}\\
    &= \cdots = \gamma_{s-1}^r(\mu_{s-1} \mu_{(s-1)+d} \cdots \mu_{(s-1)+(r-1)d}) = \gamma_0^r.
\end{align*}
Repeating this process with the $\sigma^r(a_i^*)$ gives the result. 

\eqref{rot_gam5} 
Clearly, $g^n \cdot e_0 = e_0$. If $\gcd(n,d)=1$, then 
$g^n(a_i) = \mu_i \mu_{i+1} \cdots \mu_{i-1} a_i$. Thus, by \eqref{rot_gam4}, $g^n(a_i)=a_i$. A similar argument shows that
$g^n(a_i^*)=a_i^*$. Thus $g^n=1$ and because the action is inner faithful, then $m \mid n$. But by \eqref{rot_gam1}, $n=r$ and $r \mid m$, so $m=n$.
\end{proof}

\begin{remark}
Keep the hypotheses of the previous lemma 
and suppose $n$ is odd and $d=n-2$. In this case 
\begin{align*}
    c_0 &= c_n 
        = \lambda^{(n-1)/2} \left( \frac{\mu_n\mu_{n-2} \cdots \mu_3}{\mu_{n-1}^*\mu_{n-3}^*\cdots \mu_2^*} \right) c_1
        = \lambda^{(n-1)/2} \left( \frac{\mu_n\mu_{n-2} \cdots \mu_3}{\mu_{n-1}^*\mu_{n-3}^*\cdots \mu_2^*} \right) \left( \lambda\frac{\mu_1}{\mu_0^*}c_{n-1}\right) \\
        &= \lambda^{(n+1)/2} \left( \frac{\mu_n\mu_{n-2} \cdots \mu_1}{\mu_{n-1}^*\mu_{n-3}^*\cdots \mu_0^*} \right) 
        \cdot \lambda^{(n-1)/2}
        \left( \frac{\mu_{n-1}\mu_{n-3} \cdots \mu_2}{\mu_{n-2}^*\mu_{n-4}^*\cdots \mu_1^*} \right)c_0
        = \lambda^n \left(  \frac{\mu_{n-1}\mu_{n-2} \cdots \mu_0}{\mu_{n-1}^*\mu_{n-2}^*\cdots \mu_0^*} \right)c_0.
\end{align*}
It follows that either $c_0 \neq 0$ or $c_i=0$ for all $i$.
Moreover, if $r \mid n$ then
\[\mu_{n-1}\mu_{n-2} \cdots \mu_0=\mu_{n-1}^*\mu_{n-2}^*\cdots \mu_0^*.\]
An analogous result holds in the case that $n$ is odd and $d=2$.
\end{remark}

It will be advantageous to have an upper bound on the nilpotency of $\sigma$. The next corollary shows that this index is $r$, except in the case of $n=4$ and $d=2=n-2$, which is considered at the end of the section.

\begin{lemma}\label{lem.rotsigma}
With the hypotheses of Lemma \ref{lem.rot_sigma}, if $n\neq 4$ or $d \neq 2$, then $\sigma^r(a)=0$ for any $a \in Q_1$.
\end{lemma}
\begin{proof}
If $n>4$, then at least one of $\sigma(a_i)$ and $\sigma(a_i^*)$ is zero, and therefore $\sigma^2(a_i)=\sigma^2(a_i^*)=0$ for all $i$.

If $n=3$, then the cases $d=n-2$ and $d=1$ coincide, and the cases $d=n-1$ and $d=2$ coincide. Suppose $d=n-2=1$. If $\gamma_i=0$ for $0\leq i \leq 2$, then $x^r(a)=\sigma^r(a)$ for all $a\in \cQ_1$. Since $x^r(a)=0$, it follows that $\sigma^r(a)=0$. 
On the other hand, if there exists a $j\in Q_0$ so that $\gamma_j\neq 0$, then $r= 3$. By Lemma \ref{lem.rot_sigma}, $\sigma^3(a)=0$ for all $a\in Q_1$. The argument for the $d=n-1=2$ case is similar.
\end{proof}

Recall that for the actions described above on $\kk\overline{Q}$ to descend to actions on $\Pi_Q$, it is necessary to verify \eqref{eq.mu_rel} and \eqref{eq.sigrel}. The implications of
these identities are considered in the next two lemmas.

\begin{lemma}\label{lemma.rot_desc}
Suppose $\gtaft$ acts linearly and inner faithfully on $\kk\overline{Q}$ so that $g$ acts on $\kk\overline{Q}$ via a rotation automorphism with $g\cdot e_i=e_{i+d}$ for some $0 < d \leq n-1$, and this action descends to an inner faithful action on $\Pi_Q$. Assume $n \neq 4$ or $d \neq 2$.
If there exists a vertex $j$ so that $\gamma_j\neq 0$, then for every $i\in Q_0$,
\begin{enumerate}
	\item\label{desc1} there exists an integer $z$ so that $\mu_i\mu_i^*=\lambda^ z$,
	\item\label{desc2} for every $0\leq k <\tau$ such that $k\equiv i \,\mod{\tau}$, there exists an $m\th$ root of unity, $\xi_k$, so that 
    \begin{enumerate}
    \item $\gamma_{i+1}=\xi_k\gamma_i$,
    \item $\zeta_k=\xi_k^r$, and
    \item $(\xi_k\xi_{k+1}\dotsc\xi_{k+\tau-1})^{\frac{rd}{n}}=\lambda\inv$.
    \end{enumerate}
\end{enumerate}
\end{lemma}
\begin{proof}
\eqref{desc1} By Lemma \ref{lem.rot_gamma}\eqref{rot_gam3}, for $0\leq k<d$,
\[ \mu_k\mu_{k+d}\dotsc\mu_{k+(r-1)d}\mu_k^*\mu_{k+d}^*\dotsc\mu_{k+(r-1)d}^*=1.\] 
By \eqref{eq.mu_rel}, each pair $\mu_{j'}\mu_{j'}^*$ may be replaced with $\mu_i\mu_i^*$ for $i\equiv k \,\mod{d}$ and so $(\mu_i\mu_i^*)^r=1$. It follows that $\mu_i\mu_i^*$ is an $r\th$ root of unity, so $\mu_i\mu_i^*=\lambda^z$ for some $0\leq z <r$ for all $i\in Q_0$.

\eqref{desc2} Raising both sides of the equality in \ref{lem.rot_gamma}\eqref{rot_gam1} to the $\left(\frac{m}{r}\right)\th$ power for each $i\in Q_0$ gives $\gamma_{i+1}^m=\gamma_{i}^m$. It follows that there exists an $m\th$ root of unity $\xi_i$ so that $\gamma_{i+1}=\xi_i\gamma_i$. Further, if for $0\leq k<\tau$, $k\equiv i \,\mod{\tau}$,
\[ \zeta_k\gamma_i^r=\gamma_{i+1}^r=\xi_i^r\gamma_i^r, \text{ and } \lambda\inv\gamma_i= \gamma_{i+d} =\xi_k\xi_{i+1}\dotsc\xi_{i+d-1}\gamma_i,\]
so $\zeta_i=\xi_i^r$ and $\xi_i\xi_{i+1}\dotsc\xi_{i+d-1}=\lambda\inv$. Since 
\[ \xi_i\xi_{i+1}\dotsc\xi_{i+d-1}=\lambda\inv=\xi_{i+1}\xi_{i+2}\dotsc\xi_{i+d},\]
it must be that $\xi_i=\xi_{i+d}$. Therefore, there are at most $\tau=\frac{n}{r}$ unique values of $\xi_i$, and since $\frac{\tau}{d}=\frac{rd}{n}$, the claim follows.
\end{proof}
 
\begin{lemma}\label{lem.rot_cmu}
Suppose $\gtaft$ acts linearly and inner faithfully on $\kk\overline{Q}$ so that $g$ acts on $\kk\overline{Q}$ via a rotation automorphism with $g\cdot e_i=e_{i+d}$ for some $0 < d \leq n-1$, and this action descends to an inner faithful action on $\Pi_Q$.
There exist $c,c^*\in\kk$ so that
\begin{align}\label{eq.new_sigma}
\sigma(a_i) = 
\begin{cases}
\lambda^i\mu_1\mu_2\dotsc\mu_ice_i & \text{if $d=n-1$}\\
\lambda^i\mu_1\mu_2\dotsc\mu_ica_{i-1}^* & \text{if $d=n-2$}\\
0 & \text{if $d\neq n-1,n-2$},
\end{cases} \qquad
\sigma(a_i^*) = 
\begin{cases}
(\lambda^i\mu_0^*\mu_1^*\dotsc\mu_{i-1}^*)\inv c^*e_{i+1}& \text{if $d=1$}\\
(\lambda^i\mu_0^*\mu_1^*\dotsc\mu_{i-1}^*)\inv c^*a_{i+1} & \text{if $d=2$}\\
0 & \text{if $d\neq 1,2$}.
\end{cases}
\end{align} 
If $c$ or $c^*$ is nonzero, then $\mu_0\mu_1\dotsc\mu_{n-1}=1$ or $\mu_0^*\mu_1^*\dotsc\mu_{n-1}^*=1$, respectively, $\mu_i\mu_i^*=\lambda\inv$ for all $i\in Q_0$, and if $n$ or $d$ is odd, $m=\lcm(r,n)$.
\end{lemma}
\begin{proof}
By Lemma \ref{lem.rot_sigma}, \eqref{eq.sigrel} is satisfied trivially when $2 < d < n-2$.

Suppose $n=3$ and $d=n-1=2$. Then, $\sigma(a_i)=c_ie_i$ and $\sigma(a_i^*)=c_i^*a_{i+1}^*$, and \eqref{eq.sigrel} becomes
\[  0=c_ia_i^*e_i+c_i^*\mu_ia_{i+1}a_{i+2}-c_{i+1}^*a_{i+1}a_{i+2}-c_{i+1}\mu_{i+1}^*e_{i+1}a_i^*.\]
Now, suppose $n\neq 3$ and $d=n-1$. Then, $\sigma(a_i)=c_ie_i$ but $\sigma(a_i^*)=0$, so \eqref{eq.sigrel} becomes
\begin{equation*}
    0 = a_i^*(c_ie_i) - \mu_{i+1}^* (c_{i+1}e_{i+1}) a_i^*
    = (c_i - \mu_{i+1}^* c_{i+1})a_i^*.
\end{equation*}
In either case, it follows that $c_i=\mu_{i+1}^*c_{i+1}$. From this and Lemma \ref{lem.rot_sigma}, $(\mu_1^*)\inv c_0 = c_1 = \lambda\mu_1c_0$, so $\mu_1\mu_1^* = \lambda\inv$. This now holds for all $i$ by \eqref{eq.mu_rel}.
The case $d=1$ is similar.

Suppose $n=4$ and $d=n-2=2$. Then, $\sigma(a_i)=c_ia_{i-1}^*$ and $\sigma(a_i^*)=c_i^*a_{i+1}$, and \eqref{eq.sigrel} becomes
\[ 0=c_ia_i^*a_{i-1}^*+c_i^*\mu_ia_{i+1}a_{i+2}-c_{i+1}^*a_{i+1}a_{i+2}-c_{i+1}\mu_{i+1}^*a_i^*a_{i-1}^*.\]
Now, suppose $n\neq 4$ and $d=n-2$. Then, $\sigma(a_i)=c_ia_{i-1}^*$ but $\sigma(a_i^*)=0$, so \eqref{eq.sigrel} becomes
\[ 0 = a_i^*(c_ia_{i-1}^*) - \mu_{i+1}^* (c_{i+1}ia_i^*) a_{i-1}^*
    = (c_i - \mu_{i+1}^* c_{i+1})a_i^*a_{i-1}^*.\]
In both cases, it follows that $c_i=\mu_{i+1}^*c_{i+1}$. From this and Lemma \ref{lem.rot_sigma},
\[ \lambda \mu_3(\mu_2^*)\inv c_1 = c_3 = (\mu_3^*\mu_2^*)\inv c_1,\]
so $\mu_3\mu_3^* = \lambda\inv$. This now holds for all $i$ by \eqref{eq.mu_rel}. The case $d=2$ is similar.

If $d=n-1$ or $d=n-2$, it follows that $\mu_0^*\mu_1^*\dotsc\mu_{n-1}^*=1$, and $\mu_i\mu_i^*=\lambda\inv$, so $\mu_0\mu_1\dotsc\mu_{n-1}=\lambda^{-n}=1$, since $r \mid n$. 
Similarly, if $d=1$ or $d=2$, $\mu_0^*\mu_1^*\dotsc\mu_{n-1}^*=\lambda^{-n}=1$ and $\mu_0\mu_1\dotsc\mu_{n-1}=1$. 
If $d$ or $n$ is odd, then $\kappa=n$, so $\mu_0\mu_1\dotsc\mu_{n-1}$ and $\mu_0^*\mu_1^*\dotsc\mu_{n-1}^*$ are $\left(\frac{m}{n}\right)\th$ roots of unity. Since $\lambda^{-n}$ is an $\left(\frac{r}{\gcd(r,n)}\right)\th$ root of unity, then $m=\lcm(r,n)$.
\end{proof}

The results of this section are summarized in the following theorem.

\begin{theorem}\label{thm.rot}
Fix $n > 2$. Suppose $\gtaft$ acts linearly and inner faithfully on $\kk\overline{Q}$ so that $g$ acts on $\kk\overline{Q}$ via a rotation automorphism with $g\cdot e_i=e_{i+d}$ for some 
$0 < d \leq n-1$, and this action descends to an action on $\Pi_Q$. Then, for each $i\in Q_0$, there exist scalars $\mu_i,\mu_i^*\in\kk^\times$ and $\gamma_i\in\kk$ so that
\begin{equation*} 
g\cdot a_i=\mu_ia_{i+d}, \quad g\cdot a_i^*=\mu_i^*a_{i+d}^*, \quad
x\cdot e_i = \gamma_i(e_i-\lambda\inv e_{i+d}), \quad \gamma_{i+d}=\lambda\inv\gamma_i.
\end{equation*}

(I) Suppose $\gamma_i=0$ for all $i\in Q_0$. 
Then $x \cdot a_i = \sigma(a_i)$ and $x \cdot a_i^* = \sigma(a_i^*)$ where
$\sigma(a_i),\sigma(a_i^*)$ are as in \eqref{eq.new_sigma}.

(II) Suppose there exists an $i\in Q_0$ so that $\gamma_i\neq 0$.
Then $\gamma_0 \neq 0$.

\begin{enumerate}
    \item \label{thmrot0} Case $2 < d < n-2$ and $n > 4$.
    If $i\equiv k \,\mod \tau$ for $0\leq k < \tau$, then
    \begin{align*}
        x\cdot a_i&=\lambda^{\frac{k-i}{d}}\xi_0\xi_1\dotsc\xi_{k-1}\gamma_0\left(\xi_k a_i-\lambda\inv\mu_i a_{i+d}\right),\\
        x\cdot a_i^*&=\lambda^{\frac{k-i}{d}}\xi_0\xi_1\dotsc\xi_{k-1}\gamma_0\left(a_i^*-\lambda^{z-1}\xi_k\mu_i\inv a_{i+d}^*\right),
    \end{align*}
    for some $z\in\ZZ$, $\gamma_0,\xi_0,\xi_1,\dotsc\xi_{d\tau-1}\in\kk^\times$ where the $\xi_i$ are $m\th$ roots of unity so that 
    \begin{equation*}
        (\xi_k\xi_{k+1}\dotsc\xi_{k+\tau-1})^\frac{rd}{n}=\lambda\inv\quad\text{and}\quad
        \mu_k\mu_{k+d}\dotsc\mu_{k+(r-1)d}=\xi_k^r.
    \end{equation*}
    This case also implies that $\gamma_i$ is nonzero.

    \item Case $d \in \{1,2,n-2,n-1\}$, then
    \begin{align*}
        x \cdot a_i &= \begin{cases}
            \lambda^{\mp i}\gamma_0(\lambda^{\mp 1} a_i-\lambda\inv\mu_ia_{i\pm 1})+ \sigma(a_i)
                & \text{if $d=n\pm 1$} \\
            \lambda^{\mp\frac{i}{2}} \gamma_0 (\xi_0 a_i - \lambda\inv \mu_i a_{i\pm 2}) + \sigma(a_i)
                & \text{if $d=n\pm2$, i\text{ is even}} \\
            \lambda^{\mp\frac{i-1}{2}}\gamma_0 (\lambda^{\mp 1} a_i - \lambda\inv\xi_0\mu_i a_{i\pm 2}) + \sigma(a_i) & \text{if $d=n\pm 2$, i\text{ is odd}}
        \end{cases} \\
        x \cdot a_i^* &= \begin{cases}
            \lambda^{\mp(i+1)}\gamma_0(\lambda^{\pm 1} a_i^*-\lambda\inv\mu_i^*a_{i-1}^*)+ \sigma(a_i^*)
                & \text{if $d=n\pm 1$} \\
            \lambda^{\mp \frac{i}{2}}\gamma_0 (a_i^* - \xi_0\lambda\inv \mu_i^* a_{i\pm 2}^*)+ \sigma(a_i^*)
                & \text{if $d=n\pm 2$, i\text{ is even}} \\
            \lambda^{\mp\frac{i+1}{2}} \gamma_0(\lambda^{\pm 1}\xi_0a_i^* - \lambda\inv\mu_i^* a_{i\pm2}^*)+ \sigma(a_i^*)
                & \text{if $d=n\pm2$, i\text{ is odd}}
        \end{cases}
    \end{align*}
    where $\sigma(a_i)$, $\sigma(a_i^*)$ are as in Lemma \ref{lem.rot_cmu}.    
\end{enumerate}
\end{theorem}

This section concludes by considering the special case not explicitly covered by Theorem \ref{thm.rot}.

\begin{lemma}\label{lem.n4d2}
Suppose $n=4$ and $d=2=n-2$. Then, $\sigma^r(a)=0$ for all $a\in Q_1$ if and only if one of the following holds:
\begin{enumerate}
\item $\gamma_i=0$ for all $i\in Q_0$, or
\item $\gamma_i\neq 0$ for all $i\in Q_0$, $\mu_0\mu_1\mu_2\mu_3=\mu_0^*\mu_1^*\mu_2^*\mu_3^*=1$,  and $\gamma_1^2=\mu_0\mu_2\gamma_0^2$.
\end{enumerate}
\end{lemma}
\begin{proof}
If $\gamma_i=0$ for all $i\in Q_0$, it follows immediately from \eqref{eq.gam_sig} that $\sigma^r(a)=0$ for all $a\in Q_1$. 

If there exists a $j\in Q_0$ so that $\gamma_j\neq 0$, then it must be the case that $r=2$ and $\lambda=-1$. Therefore, $\gamma_0=-\gamma_2$ and $\gamma_1=-\gamma_3$. From \eqref{eq.gam_sig} and Lemma \ref{lem.rot_cmu},
\begin{equation*}
\sigma^2(a_0)=-\frac{\mu_2}{\mu_1^*}c_0c_1^*a_0=(\gamma_0^2\mu_0\mu_2-\gamma_1^2)a_0 \quad \text{and} \quad \sigma^2(a_1^*)=(\gamma_0^2\mu_1^*\mu_3^*-\gamma_1^2)a_1^* = -\frac{\mu_2}{\mu_1^*}c_0c_1^* a_1^*,
\end{equation*}
and 
\begin{equation*}
\sigma^2(a_1)=c_0^*c_1a_1=(\gamma_1^2\mu_1\mu_3-\gamma_0^2)a_1 \quad \text{and} \quad \sigma^2(a_0^*)=(\gamma_1^2\mu_0^*\mu_2^*-\gamma_0^2)a_0^* = c_0^*c_1 a_0^*.
\end{equation*}
It is clear from the equations above that to have $\sigma^2(a)=0$ for all $a\in Q_1$ if and only if $\gamma_0\neq 0$, $\gamma_1\neq 0$, $\gamma_0^2=\mu_1\mu_3\gamma_1^2=\mu_0^*\mu_2^*\gamma_1^2$, and $\gamma_1^2=\mu_0\mu_2\gamma_0^2=\mu_1^*\mu_3^*\gamma_0^2$. The result follows.
\end{proof}

\begin{proposition}\label{prop.n4d2}
Suppose $n=4$ and $d=2=n-2$. Suppose $\gtaft$ acts linearly and inner faithfully on $\kk\overline{Q}$ so that $g\cdot e_i=e_{i+2}$, and this action descends to an action on $\Pi_Q$. 

(I) Suppose $\sigma^r(a)=0$ for all $a \in Q_1$. Then $\sigma^2(a)=0$ for all $a \in Q_1$ and the actions are as described in Theorem \ref{thm.rot} with $c=0$ or $c^*=0$.

(II) Suppose $\sigma^2(a_0) \neq 0$. Then $r=2$ and for each $i\in Q_0$ there exist scalars $\gamma_i,\mu_i,\mu_i^*,c,c^*\in\kk^\times$ so that 
\begin{align*}
g\cdot a_i=\mu_ia_{i+2}, \quad g\cdot a_i^*=\mu_i^*a_{i+2}^*, \quad
x\cdot e_i = \gamma_i(e_i+e_{i+2}), \quad \gamma_{i+2}=-\gamma_i,
\end{align*}
and
\begin{align*}
x \cdot a_i &= \gamma_{i+1}a_i + \gamma_i g \cdot a_i + (-1)^i\mu_1\mu_2\dotsc\mu_ica_{i-1}^*, \\
x \cdot a_i^* &= \gamma_ia_i^* + \gamma_{i+1} g \cdot a_i^* + (-1)^i (\mu_0^*\mu_1^*\dotsc\mu_{i-1}^*)\inv c^*a_{i+1}.
\end{align*}
Moreover, these scalars satisfy 
$\mu_0\mu_1\mu_2\mu_3=\mu_0^*\mu_1^*\mu_2^*\mu_3^*=1$, $\mu_i\mu_i^*=-1$, $\mu_0\mu_2 = \mu_1^*\mu_3^*$, and $\gamma_1^2 = \gamma_0^2\mu_0\mu_2 - \mu_3\inv cc^*$.
\end{proposition}
\begin{proof}
(I) Suppose $\sigma^r(a)=0$ for all $a \in Q_1$. 
By \eqref{eq.rot_sigma1}, either $c_i=0$ for some $i$ or $c_i^*=0$ for some $i$. 
But $c_i=\mu_{i+1}^* c_{i+1}$ and $c_i^* = \mu_i\inv c_{i+1}^*$ imply that $c_i=0$ for all $i$ or $c_i^*=0$ for all $i$, respectively. Now, in either case,
$\sigma^2(a_0) = c_0\sigma(a_3^*) = c_0c_3^*a_0 = 0$.

(II) Now assume $\sigma^2\neq 0$. By Lemma \ref{lem.n4d2}, $\gamma_i \neq 0$ for all $i$ and $r=2$. By a similar argument as in case (I), $c_i=0$ or $c_i^*=0$ for any $i$ implies that $\sigma^2(a)=0$ for all $a \in Q_1$.  Thus, $c_i,c_i^*$ are all nonzero. Set $c=c_0$, $c^*=c_0^*$. By Lemma \ref{lem.rot_cmu}, $\mu_0\mu_1\mu_2\mu_3=\mu_0^*\mu_1^*\mu_2^*\mu_3^*=1$
and $\mu_i\mu_i^*=-1$. Then
\begin{align}
\label{eq.n4d2-1}    c_ic_{i-1}^* a_i &= \sigma^2(a_i) = (\gamma_i^2 \mu_i\mu_{i+2} - \gamma_{i+1}^2) a_i \\
\label{eq.n4d2-2}   c_i^*c_{i+1} a_i^* &= \sigma^2(a_i^*) = (\gamma_{i+1}^2 \mu_i^*\mu_{i+2}^* - \gamma_i^2) a_i^*.
\end{align}
Since $\gamma_{i+2}^2=\gamma_i^2$, then 
setting $i=0$ in \eqref{eq.n4d2-1} and $i=1$ in \eqref{eq.n4d2-2} gives
\[
\gamma_0^2 \mu_1^*\mu_3^* - \gamma_1^2 
    = c_1^*c_2
    = (\mu_1\inv\mu_2\inv)(\mu_3^*\mu_0^*) c_0c_3^*
    = (\mu_1\inv\mu_2\inv)(\mu_3\inv\mu_0\inv) (\gamma_0^2 \mu_0\mu_2 - \gamma_1^2)
    = \gamma_0^2 \mu_0\mu_2 - \gamma_1^2.
\]
Thus, $\mu_0\mu_2=\mu_1^*\mu_3^*$. Similarly, $\mu_1\mu_3=\mu_0^*\mu_2^*$. 
The cases $i=0,2$ of \eqref{eq.n4d2-1} are equivalent since $\gamma_i = - \gamma_{i+2}$ and
\[ c_2c_1^* 
    = \mu_3^*\mu_0^*\mu_1\inv \mu_2\inv c_0c_3^*
    = \mu_3^*\mu_0^*\mu_1^* \mu_2^* c_0c_3^*
    = c_0c_3^*.\]
Similarly, the cases $i=1,3$ are equivalent. Finally, case $i=0$ of \eqref{eq.n4d2-1} gives
\[    \gamma_1^2 = \gamma_0^2 \mu_0\mu_2 - \mu_3\inv cc^*.\]
Substituting into case $i=1$ gives
\begin{align*}
\gamma_1^2 \mu_1\mu_3 - \gamma_0^2 - c_1c_0^*
    &= (\gamma_0^2 \mu_0\mu_2 - \mu_3\inv cc^*)\mu_1\mu_3 - \gamma_0^2 - (\mu_2^*\mu_3^*\mu_0^* cc^*) \\
    &= \mu_1 cc^* - (\mu_1^*)\inv cc^* = 0.
\end{align*}
The last equality follows because $\mu_i\mu_i^* = -1$ for all $i$.
An analogous argument holds for \eqref{eq.n4d2-2}.
\end{proof}

\subsection{Reflection actions on the preprojective algebra}
\label{sec.reflection}

Throughout this section, assume that $\gtaft$ acts linearly and inner faithfully on $\kk\overline{Q}$ and the corresponding preprojective algebra $\Pi_Q$, and that $g$ acts as a reflection according to Lemma \ref{lem.gact}. Thus, there exists $d \in \ZZ$, $0\leq d \leq n-1$,
and scalars $\mu_i,\mu_i^* \in \kk^\times$ such that for all for all $i \in Q_0$,
\[  g\cdot e_i = e_{n-(d+i)}, \qquad 
    g\cdot a_i = \mu_i a_{n-(d+i+1)}^*, \qquad 
    g\cdot a_i^* = \mu_i^* a_{n-(d+i+1)}.\]

\begin{remark}
Throughout, $j$ denotes a vertex such that the axis of reflection passes through $j$, so $g \cdot j=j$, i.e., $n-(d+j)=j$.
Similarly, $k$ denotes a vertex such that the axis of reflection passes between $k$ and $k+1$, so $g \cdot k = k+1$ and $g \cdot (k+1) = k$, i.e., $n-(d+k)=k+1$. The following lists all cases for the 
position of the axis of reflection:
\begin{enumerate}
\item If $n$ and $d$ are even, then $j=n-\frac{d}{2}$ and $j=\frac{n-d}{2}$.
\item If $n$ and $d$ are odd, then $j=\frac{n-d}{2}$ and $k=n-\frac{d+1}{2}$.
\item  If $d$ is odd and $n$ is even, $k=n-\frac{d+1}{2}$ and $k=\frac{n-d-1}{2}$.
\item If $n$ is odd and $d$ is even, then $j=n-\frac{d}{2}$ and $k=\frac{n-d-1}{2}$.
\end{enumerate}
\end{remark}

\begin{lemma}\label{lem.refl_sigma}
Suppose $\gtaft$ acts linearly and inner faithfully on $\kk\overline{Q}$ so that $g$ acts on $\kk\overline{Q}$ via a reflection automorphism with $g\cdot e_i=e_{n-(d+i)}$ for some $0< d \leq n-1$. Let $\sigma: \kk Q_0 \oplus \kk Q_1 \to \kk Q_0 \oplus \kk Q_1$ be the quiver-Taft map corresponding to this Taft action.
\begin{enumerate}
\item Suppose that the axis of reflection passes through the vertex $j$. 
Then there exists $c_j,c_j^* \in \kk$ such that
\[ \sigma(a_{j-1}) = - c_j^* a_{j-1}, \quad
\sigma(a_{j-1}^*) = - \mu_j\inv \mu_{j-1}^* c_j a_j, \quad
\sigma(a_j) = c_j a_{j-1}^*, \quad
\sigma(a_j^*) = c_j^*a_{j}^*.\]

\item Suppose that the axis of reflection passes between the vertices $k$ and $k+1$. Then there exists $c_k \in \kk$ such that
\[ \sigma(a_k)=c_k e_k, \quad \sigma(a_k^*)= \lambda \mu_k^* c_k e_{k+1}.\]
If $c_k \neq 0$, then $\lambda^2\mu_k\mu_k^*=1$.

\item If $i \neq j-1,j,k$, then $\sigma(a_i)=\sigma(a_i^*)=0$.
\end{enumerate}
\end{lemma}
\begin{proof}
By \ref{sig2},
\begin{align*}
\sigma(a_i) 
    &= e_{s(a_i)} \sigma(a_i) e_{g \cdot t(a_i)}
    = e_i \sigma(a_i) e_{n-(d+i+1)} \\
\sigma(a_i^*) 
    &= e_{s(a_i^*)} \sigma(a_i^*) e_{g \cdot t(a_i^*)}
    = e_{i+1} \sigma(a_i^*) e_{n-(d+i)}.
\end{align*}
Thus, if $\sigma(a_i) \neq 0$, then $i-1\leq n-d-i-1\leq i+1$, which means that $\frac{n-d}{2}-1\leq i \leq \frac{n-d}{2}$.
Similarly, if $\sigma(a_i^*)\neq 0$, then $i\leq n-d-i\leq i+2$, which also means that $\frac{n-d}{2}-1\leq i \leq \frac{n-d}{2}$.

Hence, if $j$ is a fixed point,
then there exists $c_{j-1}, c_{j-1}^*, c_j, c_j^* \in \kk$ such that
\[ \sigma(a_{j-1}) = c_{j-1}a_{j-1}, \quad
\sigma(a_{j-1}^*) = c_{j-1}^*a_j, \quad
\sigma(a_j) = c_j a_{j-1}^*, \quad
\sigma(a_j^*) = c_j^*a_{j}^*.\]
By \ref{sig3}, these coefficients satisfy
\begin{align*}
\mu_{j-1} c_j^* a_j^*
    &= \mu_{j-1}\sigma(a_j^*) 
    = \sigma(g\cdot a_{j-1})
    = \lambda\inv g\cdot\sigma(a_{j-1}) 
    = \lambda\inv c_{j-1}(g\cdot a_{j-1})
    = \lambda\inv c_{j-1}\mu_{j-1}a_j^* \\
\mu_{j-1}^*c_j a_{j-1}^*
    &= \mu_{j-1}^*\sigma(a_{j})
    = \sigma(g\cdot a_{j-1}^*)
    = \lambda\inv g\cdot\sigma(a_{j-1}^*)
    = \lambda\inv c_{j-1}^* (g\cdot a_j)
    = \lambda\inv c_{j-1}^*\mu_j a_{j-1}^* \\
\mu_j c_{j-1}^* a_j
    &= \mu_j \sigma(a_{j-1}^*)
    = \sigma(g \cdot a_j)
    = \lambda\inv g \cdot \sigma(a_j)
    = \lambda\inv c_j (g \cdot a_{j-1}^*)
    = \lambda\inv \mu_{j-1}^* c_j a_j \\
\mu_j^* c_{j-1} a_{j-1}
    &= \mu_j^* \sigma(a_{j-1})
    = \sigma(g \cdot a_j^*)
    = \lambda\inv g \cdot \sigma(a_j^*)
    = \lambda\inv c_j^* (g \cdot a_j^*)
    = \lambda\inv \mu_j^* c_j^* a_{j-1}.
\end{align*}
So, $c_{j-1}=\lambda c_j^* = \lambda\inv c_j^*$ and 
$c_{j-1}^* = \lambda \mu_j\inv \mu_{j-1}^*c_j = \lambda\inv \mu_j\inv \mu_{j-1}^*c_j$.
Hence, if either $c_j$ or $c_j^*$ are nonzero, then $\lambda^2=1$.
The result follows.

If $k$ and $k+1$ are swapped, then there exists $c_k,c_k^* \in \kk$ such that
\[ 
\sigma(a_k)=c_k e_k, \quad 
\sigma(a_k^*)=c_k^* e_{k+1}.
\]
Again, by \ref{sig3},
\begin{align*}
\mu_k c_k^* e_{k+1}
    &= \sigma(\mu_k a_k^*)
    = \sigma(g \cdot a_k) 
    = \lambda\inv g \cdot \sigma(a_k)
    = \lambda\inv g \cdot (c_k e_k)
    = \lambda\inv c_k e_{k+1} \\
\mu_k^* c_k e_k
    &= \sigma(\mu_k^* a_k) 
    = \sigma(g \cdot a_k^*)
    = \lambda\inv g \cdot \sigma(a_k^*) 
    = \lambda\inv g \cdot (c_k^* e_{k+1})
    = \lambda\inv c_k^* e_k.
\end{align*}
Thus, $c_k^* = (\lambda\mu_k)\inv c_k = \lambda\mu_k^* c_k$.
If follows that if $c_k \neq 0$, then $\lambda^2\mu_k\mu_k^*=1$.
\end{proof}

The coefficients $c_j,c_j^*$ intertwine with the $\gamma_i$ as explained in the next lemma.

\begin{lemma}\label{lem.refl_gam}
Suppose $\gtaft$ acts linearly and inner faithfully on $\kk\overline{Q}$ so that $g$ acts on $\kk\overline{Q}$ via a reflection automorphism with $g\cdot e_i=e_{n-(d+i)}$ for some $0< d \leq n-1$.
Suppose that the axis of reflection passes through the vertex $j$. Then 
\[
    c_j^2 = \mu_j(\mu_{j-1}^*)\inv \gamma_{j-1}^2   \qquad
    (c_j^*)^2 = \mu_j^*\mu_{j-1}\gamma_{j-1}^2.
\]
If $i \neq j-1,j$, then
\[
\gamma_{i+1}^2 
	= (\mu_i \mu_{n-(d+i+1)}^*) \gamma_i^2 
	= (\mu_i^* \mu_{n-(d+i+1)})\inv \gamma_i^2.
\]
Consequently, either $\gamma_i=0$ for all $i$, or $\gamma_i\neq 0$ for all $i$ with $g \cdot i \neq i$.
\end{lemma}
\begin{proof}
First, suppose that $i \neq j-1,j$ where $g \cdot j=j$. Then $\sigma^2(a_i)=\sigma^2(a_i^*)=0$.
This includes the case where the axis of reflection passes between $k$ and $k+1$. Then \eqref{eq.gam_sig} implies
\begin{align*}
0 &= \sigma^2(a_i) = \gamma_i^2 g^2(a_i) - \gamma_{i+1}^2 a_i
    = \left(\gamma_i^2 \mu_i \mu_{n-(d+i+1)}^* - \gamma_{i+1}^2 \right)a_i \\
0 &= \sigma^2(a_i^*) = \gamma_{i+1}^2 g^2(a_i^*) - \gamma_i^2 a_i^*
    = \left(\gamma_{i+1}^2 \mu_i^* \mu_{n-(d+i+1)} - \gamma_i^2 \right)a_i^*.
\end{align*}
If $\gamma_i \neq 0$, then $\gamma_{i+1} \neq 0$.
Thus, either $\gamma_i=0$ for all $i$, or $\gamma_i \neq 0$ for all $i$
with $g \cdot i \neq i$. Moreover,
\[
\gamma_i^2 \mu_i \mu_i^* \mu_{n-(d+i+1)}\mu_{n-(d+i+1)}^*
        = \gamma_{i+1}^2\mu_i^* \mu_{n-(d+i+1)}
        = \gamma_i^2.
\]

If $j$ is fixed, then $\gamma_j = 0$ and so by Lemma \ref{lem.refl_sigma},
\begin{align*}
(c_j^*)^2
    &= \sigma^2(a_{j-1}) 
    = \gamma_{j-1}^2 g^2(a_{j-1}) - \gamma_j^2 a_{j-1} 
    = \gamma_{j-1}^2 \mu_{j-1}\mu_j^* a_{j-1} \\
-\mu_j\inv\mu_{j-1}^*c_j^2 a_{j-1}^*
    &= \sigma^2(a_{j-1}^*) 
    = \gamma_j^2 g^2(a_{j-1}^*) - \gamma_{j-1}^2 a_{j-1}^* 
    = - \gamma_{j-1}^2 a_{j-1}^* \\
-\mu_j\inv\mu_{j-1}^*c_j^2 a_j
    &= \sigma^2(a_j) 
    = \gamma_j^2 g^2(a_j) - \gamma_{j+1}^2 a_j 
    =  - \gamma_{j+1}^2 a_j, \\
(c_j^*)^2 a_j^* 
    &= \sigma^2(a_j^*) 
    = \gamma_{j+1}^2 g^2(a_j^*) - \gamma_j^2 a_j^* 
    =  \gamma_{j+1}^2 \mu_j^*\mu_{j-1} a_j^*.
\end{align*}
The first and fourth (resp., second and third) equations given equivalent conditions through $\gamma_{j+1}=-\gamma_{j-1}$.
\end{proof}

The next lemma considers consequences of passing to the quotient.

\begin{lemma}\label{lem.refl_quot}
Suppose $\gtaft$ acts linearly and inner faithfully on $\kk\overline{Q}$ so that $g\cdot e_i=e_{n-(d+i)}$ for some $0\leq d\leq n-1$, and this action descends to an action on $\Pi_Q$.

(I) Suppose that the axis of reflection passes through the vertex $j$. Then
\[ c_j^* = -\mu_j\inv c_j.\]
If $c_j \neq 0$, then $\mu_i\mu_i^*=1$ for all $i$.

(II) Suppose that the axis of reflection passes between the vertices $k$ and $k+1$. If \eqref{eq.sigrel} holds, then
\[ c_k=-\mu_k c_k^*. \]
If $c_k \neq 0$, then $\lambda=-1$ and $\mu_i\mu_i^*=1$ for all $i$.
\end{lemma}
\begin{proof}
(I) Since $g \cdot j = j$, then $\sigma(a_{j+1})=\sigma(a_{j+1}^*)=0$ and so
\begin{align*}
0   &= a_j^*\sigma(a_j) + \sigma(a_j^*)(g \cdot a_j) - a_{j+1}\sigma(a_{j+1}^*) - \sigma(a_{j+1})(g \cdot a_{j+1}^*) \\
    &= a_j^*(c_j a_{j-1}^*) + (c_j^*a_j^*)(\mu_j a_{j-1}^*)
    = (c_j+\mu_j c_j^*)a_j^*a_{j-1}^*.
\end{align*}
This implies that $c_j=-\mu_jc_j^*$. Now
\begin{align*}
0 &= a_{j-2}^*\sigma(a_{j-2}) + \sigma(a_{j-2}^*)(g \cdot a_{j-2}) 
    - a_{j-1}\sigma(a_{j-1}^*) - \sigma(a_{j-1})(g \cdot a_{j-1}^*) \\
    &= - \left( a_{j-1}(-\mu_j\inv\mu_{j-1}^*c_j a_j) + (-c_j^*a_{j-1})(\mu_{j-1}^* a_j) \right) \\
    &= \mu_{j-1}^*\left( \mu_j\inv c_j  + c_j^* \right) a_{j-1}a_j = 0,
\end{align*}
and 
\begin{align*}
0   &= a_{j-1}^*\sigma(a_{j-1}) + \sigma(a_{j-1}^*)(g \cdot a_{j-1}) 
        - a_j\sigma(a_j^*) - \sigma(a_j)(g \cdot a_j^*) \\
    &= a_{j-1}^*(-c_j^* a_{j-1}) + (-\mu_j\inv\mu_{j-1}^*c_ja_j)(\mu_{j-1}a_j^*)    - a_j(c_j^*a_j^*) - (c_j a_{j-1}^*)(\mu_j^* a_{j-1}) \\
    &= (-c_j^*-\mu_j\inv\mu_{j-1}^*\mu_{j-1}c_j - c_j^* - \mu_j^*c_j)a_ja_j^* \\
    &= \mu_j\inv(1-\mu_{j-1}^*\mu_{j-1}) c_j a_ja_j^*.
\end{align*}
Thus, if $c_j\neq 0$, then $\mu_{j-1}\mu_{j-1}^*=1$. 
Now by \eqref{eq.mu_rel}, $\mu_i\mu_i^*=1$ for all $i$.

(II) For $i=k-1$,
\begin{align*}
0   &= a_{k-1}^*\sigma(a_{k-1}) + \sigma(a_{k-1}^*)(g \cdot a_{k-1}) 
        - a_k\sigma(a_k^*) - \sigma(a_k)(g \cdot a_k^*)  \\
    &= - \left(a_k(\lambda\mu_k^*c_k e_{k+1}) + (c_k e_k)(\mu_k^* a_k) \right)
    = -(\lambda+1)\mu_k^*c_ka_k.
\end{align*}
Thus, if $c_k \neq 0$ then $\lambda=-1$. Now for $i=k$,
\begin{align*}
0   &= a_k^*\sigma(a_k) + \sigma(a_k^*)(g \cdot a_k) 
        - a_{k+1}\sigma(a_{k+1}^*) - \sigma(a_{k+1})(g \cdot a_{k+1}^*) \\
    &= a_k^*(c_k e_k) + (\lambda\mu_k^* c_k e_{k+1})(\mu_k a_k^*)
    = (1+\lambda\mu_k^*\mu_k) c_k a_k^*.
\end{align*}
In the case that $c_k \neq 0$, then $\mu_k\mu_k^*=1$ by Lemma \ref{lem.refl_sigma} and $\mu_i\mu_i^*=1$ for all $i$ by \eqref{eq.mu_rel}.

It is clear that if $i$ is not one of the above values, then $\sigma(a_i)=\sigma(a_{i+1})=0$ and $\sigma(a_i^*)=\sigma(a_{i+1}^*)=0$. Thus, \eqref{eq.sigrel} is satisfied trivially in this case.
\end{proof}

\begin{lemma}\label{lem.refl_faithful}
Suppose $\gtaft$ acts linearly and inner faithfully on $\kk\overline{Q}$ so that $g$ acts on $\kk\overline{Q}$ via a reflection automorphism with $g\cdot e_i=e_{n-(d+i)}$ for some $0< d \leq n-1$, and this action descends to an action on $\Pi_Q$.  Then $r=2$ and $\mu_i\mu_i^*=1$.
\end{lemma}
\begin{proof}
By Proposition \ref{prop.vertact}, $\gamma_i = \lambda\inv \gamma_{g \cdot i} = \lambda\inv \gamma_{n-(d+i)}$. It follows that $\gamma_i = \lambda^{-2}\gamma_i$. Thus, if $\gamma_i \neq 0$ for some $i$, then $\lambda=-1$. Furthermore, $\#(G\cdot i)=2=r$.

Since $\gamma_i\neq 0$ for some $i$, then
$\gamma_i\neq 0$ for all $i$ with $g\cdot i \neq i$ by Lemma \ref{lem.refl_gam}. If there is a vertex $j$ such that $g \cdot j = j$, then $\gamma_{j-1}\neq 0$ and thus $c_j \neq 0$. Then by Lemma \ref{lem.refl_quot}, $\mu_i\mu_i^*=1$ for all $i$.
If there is no such vertex, then there exists a vertex $k$ such that $g \cdot k = k+1$ and $g \cdot (k+1)=k$. Thus, $\gamma_{k+1}=-\gamma_k$ and so by Lemma \ref{lem.refl_gam}, $\gamma_k^2 = \gamma_{k+1}^2 = \mu_k\mu_k^* \gamma_k^2$ and therefore $\mu_k\mu_k^*=1$. By \eqref{eq.mu_rel}, $\mu_i\mu_i^*=1$ for all $i$.

Finally, suppose that $\gamma_i=0$ for all $i$. Then there exists a vertex $k$ such that the axis passes between vertices $k$ and $k+1$, and further that $c_k,c_k^* \neq 0$. Now $\lambda=-1$ by Lemma \ref{lem.refl_gam} and $\mu_i\mu_i^*=1$ for all $i$ by Lemma \ref{lem.refl_quot}.
\end{proof}

The main theorem of this section now summarizes the results above.

\begin{theorem}\label{thm.refl}
Suppose $\gtaft$ acts linearly and inner faithfully on $\kk\overline{Q}$ so that $g$ acts on $\kk\overline{Q}$ via a reflection automorphism with $g\cdot e_i=e_{n-(d+i)}$ for some $0< d \leq n-1$, and this action descends to an action on $\Pi_Q$. Then $r=2$ and the action is described as follows:
The following hold:
\begin{enumerate}
    \item There exists scalars $\mu_i \in \kk^\times$ such that
    \[
        g \cdot e_i = e_{n-(d+i)}, \qquad
        g \cdot a_i = \mu_i a_{n-(d+i+1)}^*, \qquad
        g \cdot a_i^* = \mu_i\inv a_{n-(d+i+1)}
    \]
    where $(\mu_i\mu_{n-(d+i+1)}\inv)^{m/2}=1$ for all $i$.
    
    \item There exist scalars $\gamma_i \in \kk$ so that
    \[ x \cdot e_i = \gamma_i (e_i + e_{n-(d+i)}) \]
    and the $\gamma_i$ satisfy $\gamma_i = -\gamma_{n-(d+i)}$ (so 
    $\gamma_j=0$ when $g \cdot j = j$) and
    \[ \gamma_{i+1}^2 = \mu_i\mu_{n-(d+i+1)}\inv \gamma_i^2\]
    for $i \neq j-1,j$.
    
    \item There exists a quiver-Taft map $\sigma: \kk Q_0 \oplus \kk Q_1 \to \kk Q_0 \oplus \kk Q_1$ so that
    \begin{align*}
        x \cdot a_i &= \gamma_{i+1}a_i + \gamma_i \mu_i a_{n-(d+i+1)}^* + \sigma(a_i) \\
        x \cdot a_i^* &= \gamma_{i}a_i^* + \gamma_{i+1} \mu_i\inv a_{n-(d+i+1)} + \sigma(a_i^*)
    \end{align*}
    where the following hold:
    \begin{itemize}
    \item If $j$ is a vertex such that $g \cdot j=j$, then 
    let $c_j=\pm \gamma_{j-1} \sqrt{\mu_j\mu_{j-1}}$. 
    The map $\sigma$ satisfies
    \begin{align*} 
    \sigma(a_{j-1}) &= \mu_j\inv c_j a_{j-1}, &
        \sigma(a_j) &= c_j a_{j-1}^*, \\
    \sigma(a_{j-1}^*) &= -(\mu_j\mu_{j-1})\inv c_j a_j, &
        \sigma(a_j^*) &= -\mu_j\inv c_j a_{j}^*.
    \end{align*}
    \item If $k$ is a vertex such that $g \cdot k=k+1$ and $g \cdot (k+1)=k$, then there exists a scalar $c_k$ such that
    the map $\sigma$ satisfies
    \[ \sigma(a_k)=c_k e_k, \quad
    \sigma(a_k^*)=-\mu_k\inv c_k e_{k+1}.\]
    If $\gamma_i=0$ for all $i$, then $c_k\neq0$ for some $k$.
    \item If $i$ satisfies neither of the above conditions, then $\sigma(a_i)=\sigma(a_i^*)=0$.
    \end{itemize}
\end{enumerate}
\end{theorem}
\begin{proof}
That $r=2$ follows from Lemma \ref{lem.refl_faithful},
as does the fact that $\mu_i^*=\mu_i\inv$ for all $i$.
This implies that $\lambda=-1$.
The condition $(\mu_i\mu_{n-(d+i+1)}\inv)^{m/2}=1$ is then
necessitated by the fact that $g^m=1$.
The additional conditions on $\gamma_i$ follow from Lemma \ref{lem.refl_gam}.
The expression of $\sigma$ follows from Lemma \ref{lem.refl_sigma}.
Then Lemmas \ref{lem.refl_gam} and \ref{lem.refl_quot} reduce
the number of parameters.
\end{proof}

\section{Rotation Invariants}
\label{sec.invariants}

Given the Taft algebra actions above, it is reasonable to consider
the invariants of the actions. This section considers invariants
in the case of rotation actions. It is shown that the invariants corresponding to certain rotation actions give the center of $\Pi_Q$.
Invariants of reflection actions will be reserved for later work.

The following additional hypotheses will simplify the analysis.

\begin{hypothesis}\label{hyp.invariants}
Suppose $T=\gtaft$ acts linearly and inner faithfully on $\kk\overline{Q}$ such that $g$ acts on $\kk\overline{Q}$ via a rotation automorphism with $g\cdot e_i=e_{i+d}$ for some $0< d \leq n-1$,
and that this decends to an inner faithful action on $\Pi_Q$. 
Assume further that $r=m$, $n \geq 3$, 
and that $\sigma=0$, so that $\gamma_i \neq 0$ for all $i$.
\end{hypothesis}

With Hypothesis \ref{hyp.invariants}, the action is given explicitly as 
\begin{align*}
    g \cdot e_i &= e_{i+d}  &
        x \cdot e_i &= \gamma_i (e_i - \lambda\inv e_{i+d}) \\
    g \cdot a_i &= \mu_i a_{i+d} &
        x \cdot a_i &= \gamma_{i+1} a_i - \gamma_{i} \lambda\inv(a_{i+d}) \\
    g \cdot a_i^* &= \mu_i^* a_{i+d}^*  &
        x \cdot a_i^* &= \gamma_i a_i^* - \gamma_{i+1} \lambda\inv(a_{i+d}^*) 
\end{align*}
where $\mu_0\mu_1\cdots\mu_{n-1}=\mu_0^*\mu_1^*\cdots\mu_{n-1}^*=1$.

Let $\widehat{G}$ denote the character group of $G$.
For $\chi_\alpha \in \widehat{G}$, let $\chi_\alpha(g) = \lambda^{-\alpha}$. For a path $p$, then character-weighted orbit associated to $\chi_\alpha$ is defined as
\begin{align}\label{eq.charsum}
\phi_\alpha(p) = \sum_{i=0}^{r-1} \chi_\alpha(g^i) g^i(p).
\end{align}
Hence, 
\begin{align}\label{eq.ginv}
g \cdot \phi_\alpha(p)
	= \sum_{i=0}^{r-1} \chi_\alpha(g^i) g^{i+1}(p)
        = \lambda^\alpha \sum_{i=0}^{r-1} \chi_\alpha(g^{i+1}) g^{i+1}(p)
	= \lambda^\alpha \phi_\alpha(p).
\end{align} 

\begin{lemma}\label{lem.inv_deg0}
Assume Hypothesis \ref{hyp.invariants}.
Let $\{ e_0,\hdots,e_{\tau-1}\}$ be a complete set of coset representatives of the $g$-action on $\overline{Q}_0$. Then 
$\{ \phi_0(e_i) : i=0,\hdots,\tau-1\}$ is a $\kk$-basis for $((\Pi_Q)^T)_0$.
That is, the number of vertices of the quiver supporting $(\Pi_Q)^T$ is equal to the number of $g$-orbits.
\end{lemma}
\begin{proof}
By \eqref{eq.ginv}, it suffices to prove that $x \cdot \phi_0(e_u) = 0$. A computation gives
\begin{align*}
x \cdot \phi_0(e_u)
	&= \sum_{i=0}^{r-1} xg^i(e_u)
	= \sum_{i=0}^{r-1} \chi_1(g^i) g^i x(e_u)
	= \sum_{i=0}^{r-1} \lambda^{-i} g^i \gamma_u(e_u - \lambda\inv e_{g \cdot u}) \\
	&= \gamma_u \sum_{i=0}^{r-1} (\lambda^{-i} e_{g^i \cdot u} - \lambda^{-(i+1)} e_{g^{i+1} \cdot u})
	= 0.
\end{align*}
The result follows.
\end{proof}

For $u,v \geq 0$ and $k \in Q_0$, define
\[ p_k(u,v) = \prod_{i=0}^{u-1} a_{k+i} \cdot \prod_{j=0}^{v-1} a_{k+(u-1)-j}^*.\]
Here 
$p_k(u,0) = a_ka_{k+1}\cdots a_{k+(u-1)}$ and 
$p_k(0,v) = a_k^*a_{k-1}^* \cdots a_{k-(v-1)}^*$. 
By basic character theory, the set
\[ \{ \phi_\alpha(p_k(u,v)) \mid 0 \leq k \leq r, 0 \leq \alpha \leq s-1,u+v=\ell \} \]
is a basis of $(\Pi_Q)_\ell$.

\begin{lemma}\label{lem.phi_add}
Assume Hypothesis \ref{hyp.invariants}.
For $u,u',v,v' \geq 1$,
\[ \phi_\alpha(p_k(u,v)) \phi_\beta(p_k(u',v')) = \lambda^{(v-u)\beta}\phi_{\alpha+\beta}(p_k(u+u',v+v')).\]
\end{lemma}
\begin{proof}
Let $p=p_k(u,v)$, $q=p_k(u',v')$, and $r=p_k(u+u',v+v')$. Then
\begin{align*}
\phi_\alpha(p)\phi_\beta(q)
    &= \left(\sum_{i=0}^{r-1} \chi_\alpha(g^i) g^i(p)\right) \left(\sum_{i=0}^{r-1} \chi_\beta(g^i) g^i(q)\right) \\
    &= \sum_{i=0}^{r-1} \chi_\alpha(g^i)\chi_\beta(g^{i+u-v}) g^i(p)g^{i+u-v}(q) \\
    &= \sum_{i=0}^{r-1} \chi_{\alpha+\beta}(g^i)\chi_\beta(g^{u-v}) g^i(r) \\    
    &= \lambda^{(v-u)\beta} \phi_{\alpha+\beta}(r).
\end{align*}
This proves the result.
\end{proof}

The following notation will be useful in the sequel.
For any $k \in Q_0$, $\gamma_{k+1}^r = \zeta_k \gamma_k^r$ for some $\zeta_k \in \kk^\times$ by Lemma \ref{lem.rot_gamma}. However, since $r=m$ by assumption, then $\zeta_k=1$ for all $k$.
Fix $w_k$ such that $\gamma_{k+1}=\lambda^{w_k} \gamma_k$.

Given $u,v \geq 0$, set
\[  w_k(u,0) = w_k + w_{k+1} + \cdots + w_{k+(u-1)} \qquad
    w_k(0,v) = -(w_k + w_{k-1} + \cdots + w_{k-(v-1)})\]
and for $u,v \geq 1$ set
\[ w_k(u,v) = (w_k + w_{k+1} + \cdots + w_{k+(u-1)}) 
    - (w_{k+(u-1)} + w_{k+(u-1)-1} + \cdots + w_{k+(u-v)}).\]
The $w_k$ satisfy the following:
\begin{align}
\label{eq.inv_wk}
\begin{split}
\lambda\inv\gamma_k &= \gamma_{k+d}=\lambda^{w_k+w_{k+1}+\cdots+w_{k+(d-1)}} \gamma_k\\
\lambda\gamma_k &= \gamma_{k-d} = \lambda^{-(w_{k-1}+w_{k-2}+\cdots+w_{k-d})} \gamma_k.
\end{split}
\end{align}
That is, $w_k(d,0)=w_{k-1}(0,d)=-1$.
This implies that $w_i=w_j$ for $i,j$ on the same $g$-orbit.

\begin{lemma}\label{lem.cyc_xact}
Assume Hypothesis \ref{hyp.invariants}.
Let $u,v \geq 1$. Then
\begin{align}
\label{cyc_xact1}
x \cdot \phi_\alpha(p_k(u,0))
    &= \gamma_k (\lambda^{w_k(u,0)} - \lambda^\alpha)\phi_{\alpha+1}(p_k(u,0)) \\
\label{cyc_xact2}
x \cdot \phi_\alpha(p_k(0,v))
    &= \gamma_{k+1}(\lambda^{w_k(0,v)}-\lambda^{\alpha})\phi_{\alpha+1}(p_k(0,v)) \\
\label{cyc_xact3}
x \cdot \phi_\alpha(p_k(u,v))
    &= \gamma_k (\lambda^{w_k(u,v)} - \lambda^\alpha)\phi_{\alpha+1}(p_k(u,v)).
\end{align}
\end{lemma}
\begin{proof}
These identities can be shown inductively. First,
\begin{align*}
x \cdot \phi_\alpha(a_k)
    &= \sum_{i=0}^{r-1} \chi_\alpha(g^i) xg^i(a_k)
    = \sum_{i=0}^{r-1} \chi_\alpha(g^i) \chi_1(g^i) g^ix(a_k) \\ 
    &= \sum_{i=0}^{r-1} \chi_{\alpha+1}(g^i) g^i\left( \gamma_{k+1} a_k - \gamma_k \lambda\inv g(a_k)\right) \\
    &= \gamma_k \sum_{i=0}^{r-1}  \left( \lambda^{w_k} \chi_{\alpha+1}(g^i) g^i(a_k) - \lambda\inv \chi_{\alpha+1}(g^i) g^{i+1}(a_k)\right) \\
    &= \gamma_k \sum_{i=0}^{r-1}  \left( \lambda^{w_k} \chi_{\alpha+1}(g^{i+1}) g^{i+1}(a_k) - \lambda\inv \chi_{\alpha+1}(g^{i+1})\chi_{\alpha+1}(g\inv) (g^{i+1}(a_k)\right) \\
    &= \gamma_k \left( \lambda^{w_k}  - \lambda^\alpha \right) \phi_{\alpha+1}(a_k).
\end{align*}
Now set $p=p_k(u,0)$ and $w=w_k + w_{k+1} + \cdots + w_{k+(u-1)}$. 
Let $a=a_{k+u}$. It is easy to verify that $\phi_\alpha(p)\phi_0(a)=\phi_\alpha(pa)$. Now
\begin{align*}
x \cdot \phi_\alpha(pa)
    &= x \cdot \left(\phi_\alpha(p)\phi_0(a)\right)
    = (x \cdot \phi_\alpha(p))(g \cdot \phi_0(a))
        + \phi_\alpha(p)(x \cdot \phi_0(a)) \\
    &= \gamma_k (\lambda^w - \lambda^\alpha)\phi_{\alpha+1}(p)\phi_0(a)
        + \phi_\alpha(p) \gamma_{k+u}\left( \lambda^{w_{k+u}}  - 1 \right) \phi_1(a) \\
    &= \gamma_k \left( (\lambda^w - \lambda^\alpha) + \lambda^t (\lambda^{w_{k+u}}  - 1) \right) \phi_{\alpha+1}(p_k(u,0)) \\
    &= \gamma_k \left( \lambda^{w+w_{k+u}}  - \lambda^\alpha \right) \phi_{\alpha+1}(p_k(u,0)).
\end{align*}
This proves \eqref{cyc_xact1}. The proof of \eqref{cyc_xact2} is similar.


Now assume $u,v \geq 1$.
Let $p=p_k(u,0)$ and $q=p_{k+(u-1)}(0,v)$.
Then $\phi_\alpha(pq)=\phi_\alpha(p)\phi_0(q)$ and
\begin{align*}
x \cdot \phi_\alpha(pq)
    &= x \cdot \left(\phi_\alpha(p)\phi_0(q)\right)
    = (x \cdot \phi_\alpha(p))(g \cdot \phi_0(q))
        + \phi_\alpha(p)(x \cdot \phi_0(q)) \\
    &= \gamma_k (\lambda^{w_k(u,0)} - \lambda^\alpha)\phi_{\alpha+1}(p) \phi_0(q)
        + \phi_\alpha(p)\gamma_{k+u}(\lambda^{w_{k+(u-1)}(0,v)}-1)\phi_1(q) \\
    &= \gamma_k \left( (\lambda^{w_k(u,0)} - \lambda^\alpha) + \lambda^{w_k(u,0)} ( \lambda^{w_{k+(u-1)}(0,v)}-1)\right) \phi_{\alpha+1}(pq) \\
    &= \gamma_k \left(\lambda^{w_k(u,v)}  - \lambda^\alpha \right) \phi_{\alpha+1}(pq)
\end{align*}
where the last equality follows by noting that $w_k(u,0)+w_{k+(u-1)}(0,v)=w_k(u,v)$.
\end{proof}

\begin{proposition}\label{prop.inv_orbits}
Assume Hypothesis \ref{hyp.invariants}.
The invariant ring $(\Pi_Q)^T$ is generated by $\phi_0(e_i)$ and $\phi_0(p_i(u,v))$ such that $w_k(u,v) \equiv 0~\mod r$ and $i \in \set{1,\hdots,\tau-1}$.
\end{proposition}
\begin{proof}
By Lemma \ref{lem.inv_deg0}, 
Let $\{ e_0,\hdots,e_{\tau-1}\}$ be a complete set of coset representatives of the $g$-action on $\overline{Q}_0$. Then 
$\{ \phi_0(e_i) : i=0,\hdots,\tau-1\}$ is a $\kk$-basis for $((\Pi_Q)^T)_0$.

For $\ell \geq 1$, a basis for $(\Pi_Q)_\ell$ is
\[ \{ \phi_\alpha(p_i(u,v)) : i=1,\hdots,\tau-1, u+v=\ell\}.\]
Let $y \in (\Pi_Q)_\ell$, $\ell \geq 1$.
Write,
\[ y = \sum_{\alpha=0}^{r-1} \sum_{u+v=\ell} c_\alpha \phi_\alpha(p_i(u,v)) \]
for some $c_\alpha \in \kk$. By Lemma \ref{lem.cyc_xact},
\[ x \cdot y = \sum_{\alpha=0}^{r-1} \sum_{u+v=\ell} c_\alpha' \phi_{\alpha+1}(p_i(u,v))\]
for some $c_\alpha' \in \kk$. If $c_\alpha\neq 0$, then $c_\alpha'=0$ if and only if $u-v \equiv \alpha~\mod n$. On the other hand,
\[ g \cdot y = \sum_{\alpha=0}^{r-1} \sum_{u+v=\ell} \lambda^\alpha c_\alpha \phi_\alpha(p_i(u,v)).\]
It follows that $y$ is $g$-invariant if and only if $y$ is supported only on $\alpha \equiv 0~\mod n$. 
Finally, it is clear from Lemma \ref{lem.phi_add} that the given elements now generate the invariant ring.
\end{proof}

\begin{corollary}\label{cor.1orbit}
Assume Hypothesis \ref{hyp.invariants} and that $\gcd(d,n)=1$. 
Then the invariant ring is generated by 
$1$, $\phi_0(p_0(n,0))$, $\phi_0(p_0(0,n))$, and $\phi_0(p_0(1,1))$.
Hence, $(\Pi_Q)^T = Z(\Pi_Q)$.
\end{corollary}
\begin{proof}
Since $\gcd(d,n)=1$, then it follows from \eqref{eq.inv_wk} that $w_i=w_j$ for all $i,j$. 
Set $w=w_0$. Then $w_k(u,v)= (u-v)w~\mod r$.
Hence, by Proposition \ref{prop.inv_orbits},
$\phi_\alpha(p_i(u,v))$ is invariant if $\alpha \equiv 0~\mod r$ and $u-v \equiv 0~\mod r$. It is then clear that the given elements are invariant. By Lemma \ref{lem.phi_add}, every invariant element is generated by these.
That these elements also generate the center is well-known.
\end{proof}

\providecommand{\bysame}{\leavevmode\hbox to3em{\hrulefill}\thinspace}
\providecommand{\MR}{\relax\ifhmode\unskip\space\fi MR }
\providecommand{\MRhref}[2]{%
  \href{http://www.ams.org/mathscinet-getitem?mr=#1}{#2}
}
\providecommand{\href}[2]{#2}

\end{document}